\documentclass[review,hidelinks,onefignum,onetabnum]{siamart220329}

\nolinenumbers

\usepackage{mathrsfs}
\usepackage{amsfonts}
\usepackage{graphicx}
\usepackage{epstopdf}
\usepackage{algorithmicx}
\usepackage{algorithm}
\usepackage{algpseudocode}
\usepackage{mathtools}
\usepackage{caption,subcaption}
\newcommand{\R}{\mathbb{R}}
\newcommand{\C}{\mathbb{C}}
\newcommand{\N}{\mathbb{N}}

\usepackage{color, colortbl}
\definecolor{mywine}{rgb}{.5412,.0235,.0}
\definecolor{Gray}{gray}{0.9}
\definecolor{LightCyan}{rgb}{0.88,1,1}

\newcommand{\coolleftbrace}[2]{#1\quad\mathclap{\left\{\vphantom{\begin{matrix} #2 \end{matrix}}\right.}}
\newcommand{\coolrightbrace}[2]{\mathclap{\left.\vphantom{\begin{matrix} #1 \end{matrix}}\right\}}\quad#2}

\floatname{algorithm}{Algorithm}

\newcounter{mycounter}
\newtheorem{example}{Example}

\usepackage[T1]{fontenc}
\ifpdf
  \DeclareGraphicsExtensions{.eps,.pdf,.png,.jpg}
\else
  \DeclareGraphicsExtensions{.eps}
\fi

\newsiamremark{remark}{Remark}
\newsiamremark{hypothesis}{Hypothesis}
\crefname{hypothesis}{Hypothesis}{Hypotheses}
\newsiamthm{claim}{Claim}
\newsiamremark{assumption}{Assumption}
\newsiamremark{conjecture}{Conjecture}

\headers{On the Convergence of CROP-Anderson Acceleration Method}{N. Wan and A. Mi\k{e}dlar}

\title{On the Convergence of CROP-Anderson Acceleration Method
\thanks{
{\bfseries Funding:}
This work was supported by the National Science Foundation through the NSF CAREER Award DMS–2144181 and DMS–2324958.}}

\author{Ning Wan \thanks{Department of Mathematics, Virginia Tech, Blacksburg, VA
  (\email{wning@vt.edu, amiedlar@vt.edu})}
 \and Agnieszka Mi\k{e}dlar \footnotemark[2]
}

\usepackage{amsopn}

\makeatletter
\newcommand*{\addFileDependency}[1]{
  \typeout{(#1)}
  \@addtofilelist{#1}
  \IfFileExists{#1}{}{\typeout{No file #1.}}
}
\makeatother

\ifpdf
\hypersetup{
  pdftitle={On the Convergence of CROP-Anderson Acceleration Method},
  pdfauthor={N. Wan and A. Mi\k{e}dlar}
}
\fi

\begin{document}

\maketitle

\begin{abstract}
  Anderson Acceleration is a well-established method that allows to speed up or encourage convergence of fixed-point iterations. It has been successfully used in a variety of applications, in particular within the Self-Consistent Field (SCF) iteration method for quantum chemistry and physics computations. In recent years, the Conjugate Residual with OPtimal trial vectors (CROP) algorithm was introduced and shown to have a better performance than the classical Anderson Acceleration with less storage needed. This paper aims to delve into the intricate connections between the classical Anderson Acceleration method and the CROP algorithm. Our objectives include a comprehensive study of their convergence properties, explaining the underlying relationships, and substantiating our findings through some numerical examples. Through this exploration, we contribute valuable insights that can enhance the understanding and application of acceleration methods in practical computations, as well as the developments of new and more efficient acceleration schemes.
\end{abstract}

\begin{keywords}
  fixed-point iteration, self-consistent field iteration, acceleration method, Anderson Acceleration, CROP
\end{keywords}

\begin{MSCcodes}
  65B05, 65B99, 65F10, 65H10
\end{MSCcodes}

\section{Introduction}
\label{sec:introduction}

Consider the following problem: 
Given a function $g: \C^n \rightarrow  \C^n$ \ find \ $x \in \C^n$ \ such that
\begin{equation}
 x    \ \  =   \ g(x),  \quad
\mbox{ or alternatively } \quad
 f(x)   \ \ =  \ 0, \\
\label{eq:MainProblem}
\end{equation}
with $f(x): =  g(x) - x$. 
Obviously, a simplest method of choice to solve this problem is the fixed-point iteration 
\begin{equation}
x^{(k+1)}  = g(x^{(k)}), \ \mbox{ for all } \ k \in \N. 
\end{equation}
Unfortunately, its 
convergence is often extremely slow.

\begin{remark}
Note that in the case of a fixed-point problem \eqref{eq:MainProblem} with iteration function $g(x)$, the associated residual (error) function $f(x)$ is defined as $f(x):= g(x) - x$.
However, this choice is not in any sense universal. If the problem of interest has a specific residual (error) function associated with it, then it will usually be used to define $f(x)$. For example, in the case of C-DIIS~\cite{Pul80,Pul82} method for solving the Hartree-Fock equations, the density matrix $D$ is updated until it commutes with the associated Fock matrix $F(D)$, i.e., $F(D)D-DF(D)=0$. The new iterates
$D^{(k+1)} = g(D^{(k)})$ are computed by Roothaan SCF process. Hence, the residual (error) function $f(D)$ is defined as the commutator $f(D) = F(D)D - DF(F)$ instead of the difference between two consecutive iterates.
\end{remark}

The problem of slow (or no) convergence of a sequence of iterates has been extensively studied by researchers since the early 20th century. Aitken's delta-squared process was introduced in 1926~\cite{Ait1926} for nonlinear sequences, and since then, people have been investigating various extrapolation and convergence acceleration methods with Shanks transformation~\cite{Shanks1955,Bre2022} providing one of the most important and fundamental ideas. Introduced as a generalization of the Aitken's delta-squared process, it laid the foundations for many acceleration schemes including the $\varepsilon$-algorithm. Some notable acceleration methods, including but not limited to, are: $\varepsilon$-algorithm, which contains scalar $\epsilon$-algorithm (SEA)~\cite{Wynn1956}, vector $\epsilon$-algorithm (VEA)~\cite{Wynn1962}, topological $\epsilon$-algorithm (TEA)\cite{Bre70}, simplified TEA (STEA)\cite{Brezinski2014}; polynomial methods, which contains minimal polynomial extrapalolation (MPE)\cite{Cabay1976}, reduced rank extrapolation (RRE)\cite{Edd79}, modified minimal polynomial extrapolation (MMPE)\cite{SidFS86}. For further reading about extrapolation and acceleration methods, see~\cite{Jbilou2000,Bre00,BreRZS18, Bre20,Bre2022}.

In the following, we will consider two mixing acceleration methods: the Anderson Acceleration~\cite{And65,And19} (also referred to as Pulay mixing~\cite{Pul80,Pul82} in computational chemistry) and the \textbf{C}onjugate \textbf{R}esidual algorithm with \textbf{OP}timal trial vectors (CROP)~\cite{Ziolkowski2008,EttJ15}. 
The CROP method, introduced in~\cite{Ziolkowski2008}, is a generalization of the Conjugate Residual (CR) method~\cite[Section 6.8]{Saad2003}, which is a well-known iterative algorithm for solving linear systems. Starting with investigating broad connections between Anderson Acceleration and CROP algorithm, our goal is to understand the convergence behavior of CROP algorithm and find out for when it may serve as an alternative to the well-established Anderson Acceleration method.


\paragraph{Contributions and Outline}

In this paper, we discuss the connection between CROP algorithm and some other well-known methods, analyze its equivalence with Anderson Acceleration method and investigate convergence for linear and nonlinear problems. The specific contributions and novelties of this work are as follows:
\begin{itemize}
\item We summarize CROP algorithm and establish a unified Anderson-type framework and show the equivalence between Anderson Acceleration method and CROP algorithm.
\item We compare CROP algorithm with some Krylov subspace methods for linear problems and with multisecant methods in the general case.
\item We illustrate the connection between CROP algorithm and Anderson Acceleration method and explain the CROP-Anderson variant.
\item We investigate the situations in which CROP and CROP-Anderson algorithms work better than Anderson Acceleration method.
\item We derive the convergence results for CROP and CROP-Anderson algorithms for linear and nonlinear problems.
\item We extend CROP and CROP-Anderson algorithms to rCROP and rCROP-Anderson, respectively, by incorporating real residuals to make them work better for nonlinear problems.
\end{itemize}

\paragraph{Previous work}

Anderson Acceleration method has a long history in mathematics literature, which goes back to Anderson's 1965 seminal paper~\cite{And65}. Over the years, the method has been successfully applied to many challenging problems~\cite{CanLB00a, CanLB00b, LinLY19, Cances2020}. An independent line of research on accelerating convergence of nonlinear solvers established by physicists and chemists has led to developments of techniques such as Pulay mixing~\cite{Pul80,Pul82}, also known as the Direct Inversion of the Iterative Subspace (DIIS) algorithm,
which is instrumental in accelerating the self-consistent field iteration method in electronic structure calculations~\cite{Ni09}.


%

It is well-known that Anderson Acceleration method has connections with the Generalized Minimal Residual Method (GMRES) algorithm~\cite[Section 6.5]{Saad2003} and is categorized as a multisecant method~\cite{WalN11,RohS11,Eye96,FanS09}. The first convergence theory for Anderson Acceleration, under the assumption of a contraction mapping, appears in~\cite{TotK15}. The convergence of Anderson(1), a topic of particular interest to many researchers, is discussed separately in works such as \cite{Reb23,Hans2022,Hans2024}. Various variants of Anderson Acceleration are explored in \cite{BanSP16, SurPP19, ChuDE20, Wei23, CheK19}. The acceleration properties of Anderson Acceleration are theoretically justified in \cite{ChuDE20,EvaPRX20,Pol21,PolRX19,Reb23}. Additionally, discussions on the depth parameter can be found in \cite{Reb23,Chen22}, while the damping parameter is examined in \cite{Pol21,Pol23,Wei23}.
%
%
We further refer readers to~\cite{Jbilou2000,Bre00,BreRZS18,Bre20,Bre2022} and references therein for detailed and more comprehensive
presentation of history, theoretical and practical results on the acceleration methods and their applications.

The paper is organized as follows: \Cref{sec:introduction} briefly explains the general idea of acceleration methods, provides historical context and presents the state-of-the-art results relevant to our studies.
In \Cref{sec:background} 
we establish necessary notation and review some background material on Anderson Acceleration method and CROP algorithm. We propose a novel unified framework that allows us to illustrate explicitly the connection between the two approaches, including the role of various parameters, and perform their theoretical analysis in \Cref{sec:AndersonVsCROP}. We present convergence analysis of CROP algorithm and discuss its truncated variants in \Cref{sec:convergence}. Finally, in \Cref{sec:numerical}, we present some numerical experiments to highlight the main results.  



\section{Background}
\label{sec:background}


In this section, we will collect some essential
background information on Anderson Acceleration and CROP methods. In the discussion below, we use subscripts to denote iterates associated with a particular method, e.g., $\displaystyle x_A$ amd $\displaystyle x_C$ indicate Anderson and CROP iterates, respectively. The quantities with no subscript indicate any of the above methods.
Throughout the paper, we will talk about \emph{equivalence} of various methods, by which we mean that under additional assumptions the iterates of either algorithm can be obtained directly from the iterates of the other.

\subsection{Anderson Acceleration}

Given Anderson iterates
$\displaystyle x_A^{(k)}, k = 0, 1, \ldots$ and corresponding residual (error) vectors, e.g., $\displaystyle f_A^{(k)} := g(x_A^{(k)}) - x_A^{(k)}$, consider weighted averages of the prior iterates, i.e., 
\begin{equation}
\label{eq:WAvg}
\displaystyle
\bar x_A^{(k)} :=\sum_{i=0}^{m_A^{(k)}}\alpha_{A,i}^{(k)}x_A^{(k-m_A^{(k)}+i)} \quad \mbox{ and } \quad \bar f_A^{(k)}:=\sum_{i=0}^{m_A^{(k)}}\alpha_{A,i}^{(k)}f_A^{(k-m_A^{(k)}+i)},
\vspace{-0.17in}
\end{equation}
with weights $\displaystyle \alpha_{A,0}^{(k)}, \ldots, \alpha_{A,m_A^{(k)}}^{(k)} \in \R$ satisfying $\displaystyle \sum\limits_{i=0}^{m_A^{(k)}}\alpha_{A,i}^{(k)}=1$, a fixed depth (history or window size) parameter $\displaystyle m$ and a truncation parameter $\displaystyle m_A^{(k)} := \min\{m,k\}$.
Note that \eqref{eq:WAvg} can be written in the equivalent matrix form, i.e.,
\begin{equation}
\label{eq:WAvgM}
\displaystyle
\bar x_A^{(k)} :=X_A^{(k)}\alpha_A^{(k)} \quad  \mbox{ and } \quad \bar f_A^{(k)}:=F_A^{(k)}\alpha_A^{(k)},
\end{equation}
with $\displaystyle \mathbb{R}^{n\times(m_A^{(k)}+1)}$ matrices
$ \displaystyle
X_A^{(k)}= \big[x_A^{(k-m_A^{(k)})},\ldots,x_A^{(k)}\big], \ F_A^{(k)}= \big[f_A^{(k-m_A^{(k)})},\ldots,f_A^{(k)}\big]$, and coefficient vector
$\displaystyle \alpha_A^{(k)}=[\alpha_{A,0}^{(k)},\ldots,\alpha_{A,m_A^{(k)}}^{(k)}]^T\in\mathbb{R}^{m_A^{(k)}+1}, \|\alpha_A^{(k)}\|_1=1.$
Anderson Acceleration achieves a faster convergence than a simple fixed-point iteration
by using the past information to generate new iterates as linear combinations of previous $m_A^{(k)}$ iterates~\cite{Pul80,Pul82,WalN11}, i.e.,
\begin{equation}
\label{eq:Anderson}
\begin{aligned}
\displaystyle
x_A^{(k+1)}  & =   \bar x_A^{(k)} + \beta^{(k)} \bar f_A^{(k)} \\ 
& =  (1-\beta^{(k)}) \sum\limits_{i=0}^{m_A^{(k)}} \alpha_{A,i}^{(k)} x_A^{(k-m_A^{(k)}+i)} + \beta^{(k)} \sum\limits_{i=0}^{m_A^{(k)}} \alpha_{A,i}^{(k)}g(x^{(k-m_A^{(k)}+i)}),
\end{aligned}
\end{equation}
with given relaxation (or damping) parameters $\beta^{(k)} \in \R^{+}$ and mixing coefficients $\alpha_{A,i}^{(k)} \in \R, \ i = 0, \ldots, m_A^{(k)}$ selected to minimize the linearized residual (error) of a new iterate
within an affine space Aff$\big\{f_A^{(k-m_A^{(k)})},\ldots, f_A^{(k)}\big\}$,
i.e., obtained as a solution of the least-squares problem
\vspace{-0.1in}
\begin{equation}
\label{eq:constainedLSA}
\displaystyle
 \min_{\alpha \in \R^{m_A^{(k)}+1}}
\|F_A^{(k)}\alpha\|_{2}^{2} \ =
 \min\limits_{\alpha_{0},\ldots, \alpha_{m_A^{(k)}}} 
 \bigg\| \sum\limits_{i=0}^{m_A^{(k)}} \alpha_{i}f_A^{(k-m_A^{(k)}+i)}\bigg\|_2^2 \quad
 \mbox{s. t.} \quad \sum\limits_{i=0}^{m_A^{(k)}} \alpha_{i} = 1.
 \end{equation}
 
%
%
Note that in the case of $\beta^{(k)} = 1$ a general formulation \eqref{eq:Anderson} introduced in the original work of Anderson~\cite{And65,And19} reduces to the Pulay mixing~\cite{Pul80,Pul82}, i.e.,
\begin{equation}
\label{eq:Pulay}
 x_A^{(k+1)} = \sum_{i=0}^{m_A^{(k)}} \alpha_{A,i}^{(k)} g(x_A^{(k-m_A^{(k)}+i)}).
 \end{equation}
Therefore, Anderson Acceleration method can be summarized in 
Algorithm \ref{alg:Anderson} and is often denoted as Anderson($m$) method.

\begin{algorithm}[htp!]
\caption{Anderson Acceleration Method (of fixed depth $m$)}
\begin{algorithmic} [1]
\Require Initial Anderson iterate $x_{A}^{(0)}$, a fixed depth $m \geq 1$ and a fixed damping parameter $\beta$
\State Compute $x_A^{(1)} = g(x_A^{(0)})$
\For{$k=1,2,\ldots$ until convergence}
	\State Set truncation and relaxation parameters $m_A^{(k)}$, $\beta^{(k)}$
    \Statex \qquad \texttt{/* e.g. $m_A^{(k)} = \min\{k,m\}$, $\beta^{(k)} = \beta$ */}
	\State Set Anderson residuals $F_A^{(k)}=[f_A^{(k-m_A^{(k)})},\ldots,f_A^{(k)}]$,\ with 
    \Statex \qquad $f_A^{(i)}: = f(x_A^{(i)}) = g(x_A^{(i)}) - x_A^{(i)}$
	\State Determine mixing coefficients, i.e.,  $\alpha_A^{(k)} :=[\alpha_{A,0}^{(k)},\ldots,\alpha_{A,m_A^{(k)}}^{(k)}]^T$ that solves Problem \eqref{eq:constainedLSA}
        \State Set $x_A^{(k+1)}=(1-\beta^{(k)})\sum\limits_{i=0}^{m_A^{(k)}}\alpha_{A,i}^{(k)}x_A^{(k-m_A^{(k)}+i)}+\beta^{(k)}\sum\limits_{i=0}^{m_A^{(k)}}\alpha_{A,i}^{(k)}g(x_A^{(k-m_A^{(k)}+i)})$
\EndFor
\Ensure $x_A^{(k)}$ that solves $f(x)=0$.
\end{algorithmic}
\label{alg:Anderson}
\end{algorithm}

\noindent
Note that in what follows we consider the case of $\beta = 1$, since $\beta \neq 1$ can be reduced to the latter by setting $\displaystyle f_\beta(x)=\beta f(x)$ \ and \ $\displaystyle g_\beta(x)=x+ \beta f(x)$.

\subsection{CROP Algorithm}

Analogously, we consider iterates $\displaystyle x_C^{(k)}$, a sequence of recorded search directions $\displaystyle \Delta x_C^{(i)} := x_C^{(i+1)} - x_C^{(i)}, \ i = k - m_C^{(k)},\ldots,k-1$, and the residual (error) vectors $\displaystyle f_C^{(k)}$ generated by CROP algorithm outlined in \Cref{alg:CROP} also called CROP($m$) algorithm.
Then the new search direction $\displaystyle \Delta  x_C^{(k)}=x_C^{(k+1)}-x_C^{(k)}$ is chosen  
in the space spanned by the prior $\displaystyle m_C^{(k)}$ search directions $\displaystyle \Delta  x_C^{(i)}, i = k - m_C^{(k)}, \ldots, k-1$ and the most recent residual (error) vector $\displaystyle f_C^{(k)}$, i.e.,
\vspace{-0.05in}
\begin{eqnarray}
\label{eq:CROPiterate}
x_C^{(k+1)} & = &x_C^{(k)} + \sum\limits_{i=k-m_C^{(k)}}^{k-1} \eta_i\Delta x_C^{(i)} + \eta_k f_C^{(k)}, \nonumber
\vspace{-0.05in}
\end{eqnarray}
with some coefficients $\eta_{k-m_C^{(k)}}, \ldots, \eta_{k} \in \R$.
%

Let us assume we have carried $\displaystyle k$ steps of the CROP algorithm, i.e., we have the subspace of optimal vectors span$\displaystyle \{x_C^{(1)}, \ldots, x_C^{(k)}\}$ at hand. From the residual vector $\displaystyle f_C^{(k)}$, we can introduce a preliminary improvement of the current iterate $x_C^{(k)}$, i.e.,
\vspace{-0.05in}
\begin{equation}\label{eq:tildex}
\displaystyle
\widetilde x_C^{(k+1)} := x_C^{(k)} + f_C^{(k)}.
\end{equation} 
Now, since \eqref{eq:tildex} is equivalent to $\displaystyle f_C^{(k)} = \widetilde x_C^{(k+1)} - x_C^{(k)}$, we can find the optimal vector $\displaystyle x_C^{(k+1)}$ within the affine subspace span$\displaystyle \{x_C^{(1)}, \ldots, x_C^{(k)}, \widetilde x_C^{(k+1)}\}$, i.e.,
%
\begin{equation}
\label{eq:ufCROP}
\displaystyle
 x_C^{(k+1)} = \sum\limits_{i=0}^{m_C^{(k+1)}-1}\alpha_{C,i}^{(k+1)}x_C^{(k+1-m_C^{(k+1)}+i)} + \alpha_{C,m_C^{(k+1)}}^{(k+1)}\widetilde x_C^{(k+1)},
\vspace{-0.2in}
\end{equation}
with $\displaystyle \sum\limits_{i=0}^{m_C^{(k+1)}}\alpha_{C,i}^{(k+1)}=1$. The estimated residual (error)  $\displaystyle f_C^{(k+1)}$ corresponding to the iterate $\displaystyle x_C^{(k+1)}$ is constructed as the linear combination of the estimated residuals (errors) of each component in \eqref{eq:ufCROP} with the same coefficients, i.e.,
%
\begin{equation}
\label{eq:uffCROP}
\displaystyle
 f_C^{(k+1)} = \sum\limits_{i=0}^{m_C^{(k+1)}-1}\alpha_{C,i}^{(k+1)}f_C^{(k+1-m_C^{(k+1)}+i)} + \alpha_{C,m_C^{(k+1)}}^{(k+1)}\widetilde f_C^{(k+1)}.
\end{equation}
Note that in general, unlike for the Anderson Acceleration method, $\displaystyle f_C^{(k+1)} \neq f(x_C^{(k+1)}$.
As before, the updates $\displaystyle x_C^{(k+1)}$ and $\displaystyle f_C^{(k+1)}$ can be written in the matrix form, i.e., 
%
\begin{equation}
\label{eq:ufCROPM}
\displaystyle
x_C^{(k+1)} = X_C^{(k+1)}\alpha_C^{(k+1)} \quad \mbox{ and } \quad f_C^{(k+1)} =F_C^{(k+1)}\alpha_C^{(k+1)},
\end{equation}
with 
$\displaystyle X_C^{(k)}=\big[x_C^{(k-m_A^{(k)})},\ldots,x_C^{(k)},\widetilde x_C^{(k+1)}\big]$ and $\displaystyle F_C^{(k)}=\big[f_C^{(k-m_C^{(k)})},\ldots,f_C^{(k)},\widetilde f_C^{(k+1)}\big]$ in $\displaystyle \mathbb{R}^{n\times(m_C^{(k+1)}+1)}$, and coefficients vector $\displaystyle  \alpha_C^{(k+1)}=\big[\alpha_{C,0}^{(k)},\ldots,\alpha_{C,m_C^{(k+1)}}^{(k+1)}\big]^T$ in $\mathbb{R}^{m_C^{(k+1)}+1}$.
Minimizing the norm of the residual (error) defined in \eqref{eq:uffCROP} results in a constrained least-squares problem
\vspace{-0.1in}
\begin{equation}
\label{eq:coefCROP}
\displaystyle 
    \min\limits_{\alpha} \|F_C^{(k)}\alpha\|_2^2 = 
    \min\limits_{\alpha_0,\ldots, \alpha_{m_C^{(k+1)}}}  \bigg\| \sum_{i=0}^{m_C^{(k+1)}-1} \alpha_if_C^{(k+1-m_C^{(k+1)}+i)}+\alpha_{m_C^{(k+1)}}\widetilde f_C^{(k+1)} \bigg\|_2^2,
    \vspace{-0.1in}
\end{equation}
such that $\sum\limits_{i=0}^{m_C^{(k+1)}}\alpha_{C,i}^{(k+1)}=1$,
with a vector of mixing coefficients $\displaystyle \alpha_{C}^{(k+1)}$ as a solution.

\begin{algorithm}[htbp!]
\caption{CROP Algorithm (of fixed depth $m$)}
\begin{algorithmic} [1]
\Require Initial CROP iterate $x_C^{(0)}$, initial fixed depth $m \geq 1$
\State Compute $ \displaystyle f_C^{(0)} = f(x_C^{(0)})$
\For{$k=0,1,2,\ldots$ until convergence}
	\State Set $\widetilde x_C^{(k+1)}=x_C^{(k)}+f_C^{(k)}$ and truncation parameter $m_C^{(k+1)}$ 
 \Statex \quad \texttt{/* e.g. $m_C^{(k+1)} = \min\{k+1,m\}$ */}
	\State Set $\displaystyle F_C^{(k+1)}=\big[f_C^{(k+1-m_C^{(k)})},\ldots,f_C^{(k)}, \widetilde f_C^{(k+1)}\big]$ \ with \ $\widetilde f_C^{(k+1)}=f(\widetilde x_C^{(k+1)})$
	\State Determine mixing coefficients, i.e., $\displaystyle \alpha_C^{(k+1)}=\big[\alpha_{C,0}^{(k+1)},\ldots,\alpha_{C,m_C^{(k+1)}}^{(k+1)}\big]^T$ that \Statex \quad \ solves Problem \eqref{eq:coefCROP}
        \vspace{-0.05in}
	\State Set $\displaystyle  x_C^{(k+1)}=\sum\limits_{i=0}^{m_C^{(k+1)}-1}\alpha_{C,i}^{(k+1)}x_C^{(k+1-m_C^{(k+1)}+i)}+\alpha_{C,m_C^{(k+1)}}^{(k+1)}\widetilde x_C^{(k+1)}$
        \vspace{-0.05in}
        \State Set CROP residuals $\displaystyle  f_C^{(k+1)}=\sum\limits_{i=0}^{m_C^{(k+1)}-1}\alpha_{C,i}^{(k+1)}f_C^{(k+1-m_C^{(k+1)}+i)}+\alpha_{C,m_C^{(k+1)}}^{(k+1)}\widetilde f_C^{(k+1)}$
\EndFor
\Ensure $x_C^{(k+1)}$ that solves $f(x)=0$.
\end{algorithmic}
\label{alg:CROP}
\end{algorithm}
\noindent
In \Cref{alg:CROP} the superscript of the iterates in step $k$ is chosen as $k+1$ instead of $k$ for a reason. In this way, the case of $m_C^{(k)} =  k$ indicates no truncation. Also, as we will see in \Cref{sec:AndersonVsCROP}, this choice enables us to understand the correspondence between classical Anderson Acceleration method and CROP algorithm.


\subsection{The Least-Squares Problem}
For both the classical Anderson Acceleration method and CROP algorithm, obtaining mixing coefficients $\displaystyle  \alpha_0^{(k)},\ldots, \alpha_{m^{(k)}}^{(k)}$ requires solving constrained least-squares problems, i.e., \eqref{eq:constainedLSA} and \eqref{eq:coefCROP}, of the same general form 
\begin{equation}
\label{eq:constainedLS}
\displaystyle 
 \min\limits_{\alpha_0^{(k)},\ldots, \alpha_{m^{(k)}}^{(k)}} 
 \bigg\| \sum\limits_{i=0}^{m^{(k)}} \alpha_{i}^{(k)}f^{(k-m^{(k)}+i)}\bigg\|_2^2 \quad \mbox{ such that } \quad
 \sum\limits_{i=0}^{m^{(k)}} \alpha_i^{(k)} = 1,
 \vspace{-0.1in}
 \end{equation}
defined within the affine subspace Aff$\displaystyle  \big\{x^{(k-m^{(k)})},\ldots, x^{(k)}\big\}$. 
In general, \eqref{eq:constainedLS} can be solved by the method of Lagrange multipliers~\cite{Pul80}. By changing the barycentric coordinates into the affine frame, \eqref{eq:constainedLS} becomes
\vspace{-0.05in}
\begin{equation}
\label{eq:unconstrainedLS}
     \min\limits_{\gamma_1^{(k)},\ldots, \gamma_{m^{(k)}}^{(k)}} 
 \bigg\|f^{(k)}- \sum\limits_{i=1}^{m^{(k)}} \gamma_{i}^{(k)}\Delta f^{(k-m^{(k)}+i)}\bigg\|_2^2,
 \vspace{-0.05in}
\end{equation}
where $\displaystyle \Delta f^{(i)}=f^{(i+1)}-f^{(i)}$. Also, $\displaystyle \alpha^{(k)}$ and $\displaystyle \gamma^{(k)}$ can be transformed, i.e., $\displaystyle \alpha_0^{(k)}=\gamma_{1}^{(k)}$, $\displaystyle  \alpha_i^{(k)}=\gamma_{i+1}^{(k)}-\gamma_{i}^{(k)}, i=1,\ldots,m^{(k)}-1$, $\displaystyle  \alpha_{m^{(k)}}^{(k)}=1-\gamma_{m^{(k)}}^{(k)}$.
Now, \eqref{eq:unconstrainedLS} can be solved by minimal equations. Let $\mathscr{F}^{(k)} :=\big[\Delta f^{(k-m^{(k)})},\ldots,\Delta f^{(k-1)}\big]\in \mathbb{R}^{n\times m^{(k)}}.$
Then, $\bar f^{(k)}=f^{(k)}-\mathscr{F}^{(k)}\gamma^{(k)}$
and the solution of \eqref{eq:unconstrainedLS} is given as
\vspace{-0.1in}
\begin{equation}\label{eq:gamma}
\gamma^{(k)}=\bigg[(\mathscr{F}^{(k)})^T\mathscr{F}^{(k)}\bigg]^{-1}(\mathscr{F}^{(k)})^Tf_k.
\vspace{-0.05in}
\end{equation}
Using \eqref{eq:unconstrainedLS} and \eqref{eq:gamma} we can write the update of Anderson Acceleration and CROP method as
\begin{equation}
\label{eq:aufxm}
\displaystyle 
 x_A^{(k+1)} = x_A^{(k)}+\beta f_A^{(k)} -(\mathscr{X}_A^{(k)}+\beta\mathscr{F}_A^{(k)})[(\mathscr{F}_A^{(k)})^T\mathscr{F}_A^{(k)}]^{-1}(\mathscr{F}_A^{(k)})^Tf_A^{(k)},
\end{equation}
and
\begin{equation}\label{eq:cuf}
\displaystyle 
 x_C^{(k+1)} = \widetilde x_C^{(k)} -\mathscr{X}_C^{(k+1)}[(\mathscr{F}_C^{(k+1)})^T\mathscr{F}_C^{(k+1)}]^{-1}(\mathscr{F}_C^{(k+1)})^T\widetilde f_C^{(k+1)},
\end{equation}
where
$$\displaystyle \mathscr{X}_A^{(k)}=\big[\Delta x_A^{(k-m_A^{(k)})},\ldots,\Delta x_A^{(k-1)}\big] ,\quad \mathscr{F}_A^{(k)}=\big[\Delta f_A^{(k-m_A^{(k)})},\ldots,\Delta f_A^{(k-1)}\big]$$
$$\displaystyle  \mathscr{X}_C^{(k+1)}=\big[\Delta x_C^{(k+1-m_C^{(k+1)})},\ldots,\Delta x_C^{(k-1)},\widetilde x_C^{(k+1)}-x_C^{(k)}\big] ,$$
$$\mathscr{F}_C^{(k+1)}=[\Delta f_C^{(k+1-m_C^{(k+1)})},\ldots,\Delta f_C^{(k-1)},\widetilde f_C^{(k+1)}-f_C^{(k)}].$$

In the practical implementation, the basis $\displaystyle \Delta f^{(k)}$ are often orthogonalized using a QR factorization which also enables the least-squares problem use the information from the previous iteration steps~\cite{WalN11,NiW10}. 
In both \Cref{alg:Anderson} and \cref{alg:CROP}, the least-squares problem (line 5) is solved using the QR factorization, however, if the least-squares problem is small, explicit pseudoinverse formulation of \eqref{eq:aufxm} and \eqref{eq:cuf} is used instead.

\subsection{The damping parameter $\beta$}

It is well-known that a good choice of the damping parameter $\beta_k$ significantly influences the convergence of Anderson Acceleration method. In the case of a fixed damping, i.e., $\beta_k=\beta$, if $\beta\neq 1$, the update formula \eqref{eq:Pulay} has the form
\vspace{-0.15in}
\begin{equation}
\nonumber
\displaystyle
x_A^{(k+1)}=(1-\beta)\sum_{i=0}^{m_A^{(k)}}\alpha_i^{(k)}x_A^{(k-m_A^{(k)}+i)}+\beta\sum_{i=0}^{m_A^{(k)}}\alpha_i^{(k)}g(x_A^{(k-m_A^{(k)}+i)}).
\end{equation}
Defining \ $\displaystyle g_\beta(x):= (1-\beta)x+\beta g(x)=x+\beta (g(x)-x)=x+\beta f(x)$ \ yields an equivalent formulation
\vspace{-0.25in}
\begin{equation}
\nonumber
\displaystyle
x_A^{(k+1)}=\sum_{i=0}^{m_A^{(k)}}\alpha_i^{(k)}g_\beta(x_{k-m_A^{(k)}+i}),
\vspace{-0.05in}
\end{equation}
which has the same form as the iterates in~\Cref{alg:Anderson}. Hence, Anderson Acceleration method with a fixed depth parameter $\beta$ can be regarded
as running~\Cref{alg:Anderson} with a new fixed-point iteration function $g_\beta$.
It is worth mentioning, that the corresponding residual (error) $f_\beta(x)=g_\beta(x)-x=\beta(g(x)-x)=\beta f(x)$, 
although different than 
$f(x)$, has the same zeros as $f(x)$.

According to \cite[Propsition 4.3]{EvaPRX20}, the residual $f
_A^{(k)}$  at step $k$ is bounded by $\beta_k$ in the following way:
\begin{equation}
    \|f_A^{(k+1)}\|_2\le \theta_{k+1}\bigg(\Big(\big(1-\beta_k\big)+L_g\beta_k\Big)\bigg)\|f_A^{(k)}\|+\sum_{j=0}^{m_A}\mathcal{O}\big(\|f_A^{(k-j)}\|_2^2\big),
\end{equation}
where $L_g$ is the Lipchitz constant of the mapping $g$, and $\theta_{k+1}=\|\bar f_A^{(k+1)}\|_2/\|f_A^{(k+1)}\|_2$.

\subsection{The truncation parameter $m^{(k)}$}

In general, the subspace truncation parameter $\displaystyle m^{(k)} \geq 1$ ($\displaystyle m_A^{(k)}$ or $\displaystyle m_C^{(k)}$) determines the dimension of the search space for the next trial vector ($\displaystyle x_A^{(k+1)}$ or $x_C^{(k+1)}$), e.g., the size of the least-squares problem \ref{eq:constainedLS}. Hence, the affine subspace in iteration step $\displaystyle k$ involves $\displaystyle m^{(k)}+1$ vectors.
Obviously, $m^{(k)} = 0$ corresponds to the fixed-point iteration method.

Note that for Anderson Acceleration method and CROP algorithm, the situations are slightly different. In Anderson Acceleration, at step $\displaystyle k$ iterate $\displaystyle x_A^{(k+1)}$ is computed from iterates $\displaystyle x_A^{(k-m^{(k)})},\ldots,x_A^{(k)}$, whereas in CROP algorithm, $\displaystyle x_C^{(k+1)}$ is computed using vectors $\displaystyle x_C^{(k+1-m^{(k+1)})},\ldots,x_C^{(k)},\widetilde x_C^{(k+1)}$. Therefore, we can immediately see that Anderson(1) is a mixing method that needs the historical information in every step, while CROP(1) is a fixed-point iteration, i.e., in Anderson(1) step to compute $\displaystyle x_A^{(k+1)}$ we need access to $\displaystyle x_A^{(k-1)}$ and $\displaystyle x_A^{(k)}$, whereas in the case of CROP(1) we only need $\displaystyle x_C^{(k)}$ to determine next iterate $\displaystyle x_C^{(k+1)}$.

The proper choice of $m^{(k)}$ is very important as it affects the convergence speed, time and overall complexity of these methods. For the fixed depth methods, a fixed parameter $m$ (e.g. $m \le 5$) is set before the iteration starts and the truncation parameter is chosen at each iteration step $k$ to be $m^{(k)}=\min\{k,m\}$.
For further discussions regarding various choices of $m^{(k)}$ see \cite{Reb23,Hans2022}. Values $m\le 5$ are often used in practice, in particular $m=1$. A larger $m$ is usually unnecessary and will cause the least-squares problem to be hard to solve. In what follows, we will briefly talk about values of $m$ for CROP method.

Restarting and adaptation are used to maintain the dimension of the subspace by choosing parameter $m^{(k)}$ in each step~\cite{ChuDE20}. Also, some techniques like filtering~\cite{Pol23} do not use the most recent $m^{(k)}$ trial vectors and residuals. However, they usually need to store a list of trial vectors and residuals of size $m^{(k)}$.

\section{Anderson Acceleration vs CROP Method}
\label{sec:AndersonVsCROP}
The aim of this section is to establish a unified framework which will enable us to show that Anderson Acceleration method and CROP algorithm are equivalent. By showing the equivalence between the full versions of these two methods, some new variants, i.e., CROP-Anderson and rCROP, can be developed
and modified to get the real residuals.


\begin{theorem}
\label{thm:AequivC}
Let us consider applying Anderson Acceleration method ($\beta_k=1$) and CROP algorithm to the nonlinear problem $f(x) = 0$ with initial values $x_A^{(0)} = x_C^{(0)}$ and no truncation ($m^{(k)}=k$). Then, for $k=0,1,\ldots$ until convergence ($f_C^{(0)} \neq0$)
\[
x^{(k)}_{C} = \bar x_A^{(k)}, \ f^{(k)}_{C} = \bar f_A^{(k)}
\quad \mbox{ and } \quad
x^{(k+1)}_{A} = \widetilde x^{(k+1)}_{C}, \ f^{(k+1)}_{A} = \widetilde f^{(k+1)}_{C},
\]
with $\bar x^{(k)}_A$, $x^{(k+1)}_A$ defined as in \eqref{eq:WAvg} and \eqref{eq:Anderson},
and $\widetilde x^{(k+1)}_{C}$, $x^{(k)}_C$ as in \eqref{eq:tildex} and \eqref{eq:ufCROP}.
\end{theorem}

\begin{proof}
The proof follows by induction. Since $x_A^{(0)} = x_C^{(0)}$, then for $k=0$,
$x_C^{(0)}=\bar x_A^{(0)}$, $f_C^{(0)}=f(x_C^{(0)})=f(x_A^{(0)})=f_A^{(0)}=\bar f_A^{(0)}$.
Hence, $\widetilde x_C^{(1)}=x_C^{(0)}+f_C^{(0)}=x_A^{(0)}+f_A^{(0)}=x_A^{(1)}$ and
$\widetilde f_C^{(1)}=f(\widetilde x_C^{(1)})=f(x_A^{(1)})=f_A^{(1)}$. 
Assume that
for all $\ell \le k$, $x_C^{(\ell)} = \bar x_A^{(\ell)}$, $f_C^{(\ell)}=\bar f_A^{(\ell)}$, $\widetilde x_C^{(\ell+1)}=x_A^{(\ell+1)}$ and $\widetilde f_C^{(\ell+1)}=f_A^{(\ell+1)}$.
Then, for $k+1$
\begin{equation}
\nonumber
\begin{split}
f_C^{(k+1)}&=\sum_{i=0}^{m^{(k+1)}-1}\alpha_{C,i}^{(k+1)}f_C^{(k+1-m_C^{(k+1)}+i)}+\alpha_{C,m^{(k+1)}}^{(k+1)}\widetilde f_C^{(k+1)}\\
    &=\sum_{i=0}^{k}\alpha_{C,i}^{(k+1)}\bar f_A^{(i)}+\alpha_{C,k+1}^{(k+1)}f_A^{(k+1)}=\sum_{i=0}^{k}\alpha_{C,i}^{(k+1)}\sum_{j=0}^{i} \alpha_{A,j}^{(i)} f_A^{(j)}+\alpha_{C,k+1}^{(k+1)}f_A^{(k+1)}\\
    &=\sum_{j=0}^{k}\sum_{i=j}^{k}\alpha_{C,i}^{(k+1)} \alpha_{A,j}^{(i)} f_A^{(j)}+\alpha_{C,k+1}^{(k+1)}f_A^{(k+1)}.
\end{split}
\end{equation}
Let $\displaystyle \widehat\alpha_j^{(k+1)}=\sum\limits_{i=j}^{k}\alpha_{C,i}^{(k+1)} \alpha_{A,j}^{(i)}$, $j=1,\ldots,k$, and $\displaystyle \widehat\alpha_{k+1}^{(k+1)}=\alpha_{C,k+1}^{(k+1)}$, then
\begin{equation}
\displaystyle
\nonumber
f_C^{(k+1)}=\sum_{j=0}^{k+1}\widehat\alpha_{j}^{(k+1)} f_A^{(j)}
\quad \mbox{ and } \quad \sum_{j=0}^{k+1}\widehat\alpha_j^{(k+1)}=\sum_{i=0}^{k+1}\alpha_{C,i}^{(k+1)}=1.
\end{equation}
\noindent
Since $\displaystyle f_C^{(k+1)}=\sum\limits_{i=0}^{k+1}\widehat\alpha_{i}^{(k+1)} f_A^{(i)}$ \ and \
$\displaystyle \bar f_A^{(k+1)} = \sum\limits_{i=0}^{k+1} \alpha_{A,i}^{(k+1)} f_A^{(i)}$ are the solutions of the least-squares problem in the same affine space Aff$\{f_A^{(0)},\ldots,f_A^{(k+1)}\}$, we know that $\widehat\alpha_{i}^{(k+1)}=\alpha_{A,i}^{(k+1)}$, and $f_C^{(k+1)}=\bar f_A^{(k+1)}$. Also, 
\begin{eqnarray}
\displaystyle x_C^{(k+1)} & = & \sum\limits_{i=0}^{k+1}\widehat\alpha_{i}^{(k+1)} x_A^{(i)} = \sum\limits_{i=0}^{k+1} \alpha_{A,i}^{(k+1)} f_A^{(i)}=\bar x_A^{(k+1)},\nonumber \\
\widetilde x_C^{(k+2)} & = &x_C^{(k+1)}+f_C^{(k+1)}=\bar x_A^{(k+1)}+\bar f_A^{(k+1)}=x_A^{(k+2)}, \nonumber
\end{eqnarray}
and 
$\displaystyle \widetilde f_C^{(k+2)}=f(\widetilde x_C^{(k+2)})=f(x_A^{(k+2)})=f_A^{(k+2)}$.
Therefore by induction $\displaystyle x_C^{(k)} = \bar x_A^{(k)}$, \ $\displaystyle f_C^{(k)}=\bar f_A^{(k)}$, \ $\displaystyle \widetilde x_C^{(k+1)}=x_A^{(k+1)}$ and $\displaystyle \widetilde f_C^{(k+1)}=f_A^{(k+1)}$ for all $k\in\mathbb{N}$.
\end{proof}

\begin{remark}
    Since the CROP residual $f_C^{(k)}$ may become exactly $0$, the \Cref{alg:CROP} can break down. However, before the actual breakdown occurs, \Cref{thm:AequivC} holds. If the CROP algorithm breaks down at step $k$ with $f_C^{(k)} = 0$, it stops, whereas in the case of Anderson Acceleration $x_A^{(k+1)} = x_A^{(k)}$ and the stagnation occurs.
\end{remark}
To illustrate a connection between Anderson Acceleration method and CROP algorithm, in~\Cref{fig:AAvsCROP},
we consider one and a half step of Anderson Acceleration from iterate $\displaystyle x_A^{(k)}$ to $\displaystyle \bar x_A^{(k+1)}$. By the equivalence of the least-squares problems in ~\eqref{eq:constainedLSA} and \eqref{eq:coefCROP}, block (II) in~\Cref{fig:AAvsCROP} is equivalent to a single step of CROP algorithm. Changing the weighted averages of Anderson Acceleration to CROP type averages enables us to obtain a different scheme illustrated in~\Cref{fig:CROPvsAA}. 
\begin{figure}[htbp!]
\[
\begin{matrix}
\coolleftbrace{(I)}{
x_A^{(k)} = \bar x_A^{(k-1)}+\bar f_A^{(k-1)}\\
\downarrow\\
f_A^{(k)} :=f(x_A^{(k)})\\
\downarrow\\
\mbox{Find } \alpha_A^{(k)} \mbox{ that minimizes } \|\bar f_A^{(k)}\|_2 \\
\mbox{ with } \bar f_A^{(k)} :=\sum\limits_{i=0}^{m_A^{(k)}}\alpha_{A,i}^{(k)}f_A^{(k-m_A^{(k)}+i)}\\
\downarrow\\
\bar x_A^{(k)}=\sum\limits_{i=0}^{m_A^{(k)}}\alpha_{A,i}^{(k)}x_A^{(k-m_A^{(k)}+i)}\\
\downarrow\\
x_A^{(k+1)}=\bar x_A^{(k)}+\bar f_A^{(k)}
}\\
\vphantom{\downarrow}\\
\vphantom{f_A^{(k+1)}=f(x_A^{(k+1)})}\\
\vphantom{\downarrow}\\
\vphantom{\mbox{Find } \alpha_A^{(k)} \mbox{ that minimizes } \|\bar f_A^{(k)}\|_2} \\ \vphantom{\mbox{ with } \bar f_A^{(k)} :=\sum\limits_{i=0}^{m_A^{(k)}}\alpha_{A,i}^{(k)}f_A^{(k-m_A^{(k)}+i)}}\\
\vphantom{\downarrow}\\
\vphantom{\bar x_A^{(k+1)}=\sum\limits_{i=0}^{m_A^{(k+1)}}\alpha_{A,i}^{(k+1)}x_A^{(k+1-m_A^{(k+1)}+i)}}
\end{matrix}%
\begin{matrix}
x_A^{(k)} = \bar x_A^{(k-1)}+\bar f_A^{(k-1)}\\
\downarrow\\
f_A^{(k)} :=f(x_A^{(k)})\\
\downarrow\\
\mbox{Find } \alpha_A^{(k)} \mbox{ that minimizes } \|\bar f_A^{(k)}\|_2 \\ \mbox{ with } \bar f_A^{(k)} =\sum\limits_{i=0}^{m_A^{(k)}}\alpha_{A,i}^{(k)}f_A^{(k-m_A^{(k)}+i)}\\
\downarrow\\
\bar x_A^{(k)}=\sum\limits_{i=0}^{m_A^{(k)}}\alpha_{A,i}^{(k)}x_A^{(k-m_A^{(k)}+i)}\\
\downarrow\\
x_A^{(k+1)}=\bar x_A^{(k)}+\bar f_A^{(k)}\\
\downarrow\\
f_A^{(k+1)}:=f(x_A^{(k+1)})\\
\downarrow\\
\mbox{Find } \alpha_A^{(k+1)} \mbox{ that minimizes } \|\bar f_A^{(k+1)}\|_2 \\ \mbox{ with } \ \bar f_A^{(k+1)}=\sum\limits_{i=0}^{m_A^{(k+1)}}\alpha_{A,i}^{(k+1)}f_A^{(k+1-m_A^{(k+1)}+i)}\\
\downarrow\\
\bar x_A^{(k+1)}=\sum\limits_{i=0}^{m_A^{(k+1)}}\alpha_{A,i}^{(k+1)}x_A^{(k+1-m_A^{(k+1)}+i)}
\end{matrix}
\begin{matrix}
\vphantom{x_A^{(k)} = \bar x_A^{(k-1)}+\bar f_A^{(k-1)}}\\
\vphantom{\downarrow}\\
\vphantom{f_A^{(k)} :=f(x_A^{(k)})}\\
\vphantom{\downarrow}\\
\vphantom{\mbox{Find } \alpha_A^{(k)} \mbox{ that minimizes } \|\bar f_A^{(k)}\|_2}\\ \vphantom{\mbox{ with } \bar f_A^{(k)} :=\sum\limits_{i=0}^{m_A^{(k)}}\alpha_{A,i}^{(k)}f_A^{(k-m_A^{(k)}+i)}}\\
\vphantom{\downarrow}\\
\coolrightbrace{\bar x_A^{(k)}=\sum\limits_{i=0}^{m_A^{(k)}}\alpha_{A,i}^{(k)}x_A^{(k-m_A^{(k)}+i)}\\
\downarrow\\
x_A^{(k+1)}=\bar x_A^{(k)}+\bar f_A^{(k)}\\
\downarrow\\
f_A^{(k+1)}:=f(x_A^{(k+1)})\\
\downarrow\\
\mbox{Find } \alpha_A^{(k+1)} \mbox{ that minimizes } \|\bar f_A^{(k+1)}\|_2 \\ \mbox{ with } \ \bar f_A^{(k+1)}=\sum\limits_{i=0}^{m_A^{(k+1)}}\alpha_{A,i}^{(k+1)}f_A^{(k+1-m_A^{(k+1)}+i)}\\
\downarrow\\
\bar x_A^{(k+1)}=\sum\limits_{i=0}^{m_A^{(k+1)}}\alpha_{A,i}^{(k+1)}x_A^{(k+1-m_A^{(k+1)}+i)}
}{(II)}
\end{matrix}%
\]
\caption{Anderson Acceleration in Anderson notation.}
\label{fig:AAvsCROP}
\vspace{-0.1in}
\end{figure}

\begin{figure}[ht]
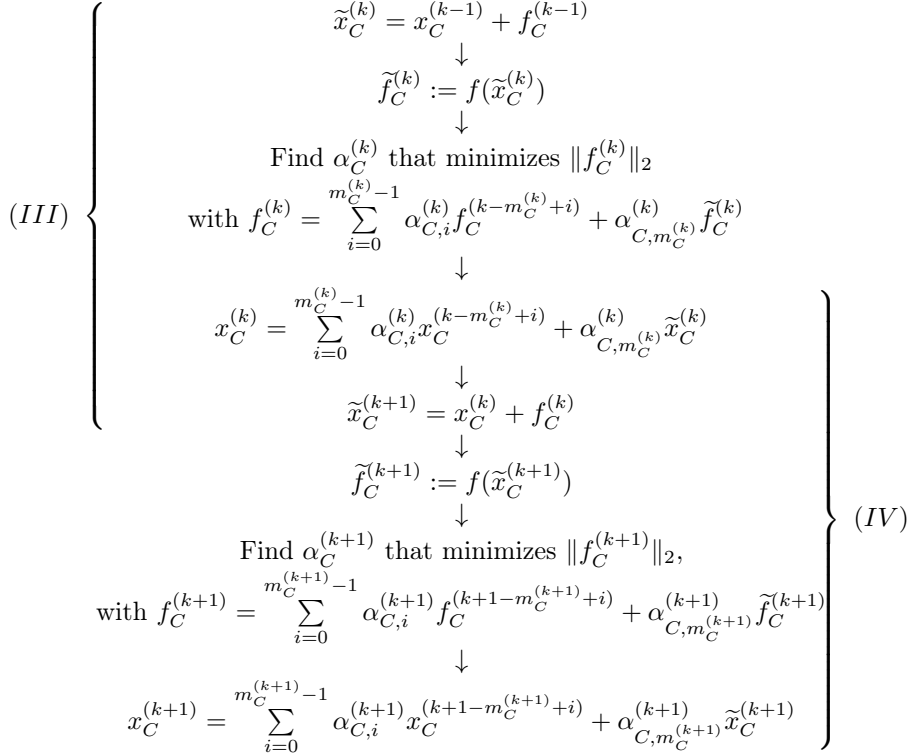

\[
\begin{matrix}
\coolleftbrace{(III)}{\widetilde x_C^{(k)}=x_C^{(k-1)}+f_C^{(k-1)}\\
\downarrow\\
\widetilde f_C^{(k)}=f(\widetilde x_C^{(k)})\\
\downarrow\\
\mbox{Find } \alpha_C^{(k)} \mbox{ that minimizes } \|f_C^{(k)}\|_2 \\
\mbox{ with }  f_C^{(k)}=\sum\limits_{i=0}^{m_C^{(k)}-1}\alpha_{C,i}^{(k)} f_C^{(k-m_C^{(k)}+i)}+\alpha_{C,m_C^{(k)}}^{(k)} \widetilde f_C^{(k)}\\
\downarrow\\
x_C^{(k)}=\sum\limits_{i=0}^{m_C^{(k)}-1}\alpha_{C,i}^{(k)} x_C^{(k-m_C^{(k)}+i)}+\alpha_{C,m_C^{(k)}}^{(k)} \widetilde x_C^{(k)}\\
\downarrow\\
\widetilde x_C^{(k+1)}=x_C^{(k)}+f_C^{(k)}
}\\
\vphantom{\downarrow}\\
\vphantom{\widetilde f_C^{(k+1)}=f(\widetilde x_C^{(k+1)})}\\
\vphantom{\downarrow}\\
\vphantom{\mbox{Find } \alpha_C^{(k+1)} \mbox{ that minimizes } \|f_C^{(k+1)}\|_2, }\\
\vphantom{\mbox{with } f_C^{(k+1)}=\sum\limits_{i=0}^{m_C^{(k+1)}-1}\alpha_i^{(k+1)} f_C^{(k+1-m_C^{(k+1)}+i)}+\alpha_{m_C^{(k+1)}}^{(k+1)} \widetilde f_C^{(k+1)}}\\
\vphantom{\downarrow}\\
\vphantom{x_C^{(k+1)}=\sum\limits_{i=0}^{m_C^{(k+1)}-1}\alpha_{C,i}^{(k+1)} x_C^{(k+1-m_C^{(k+1)}+i)}+\alpha_{C,m_C^{(k+1)}}^{(k+1)} \widetilde x_C^{(k+1)}}
\end{matrix}%
\begin{matrix}
\widetilde x_C^{(k)}=x_C^{(k-1)}+f_C^{(k-1)}\\
\downarrow\\
\widetilde f_C^{(k)}:=f(\widetilde x_C^{(k)})\\
\downarrow\\
\mbox{Find } \alpha_C^{(k)} \mbox{ that minimizes } \|f_C^{(k)}\|_2 \\ \mbox{ with }  f_C^{(k)}=\sum\limits_{i=0}^{m_C^{(k)}-1}\alpha_{C,i}^{(k)} f_C^{(k-m_C^{(k)}+i)}+\alpha_{C,m_C^{(k)}}^{(k)} \widetilde f_C^{(k)}\\
\downarrow\\
x_C^{(k)}=\sum\limits_{i=0}^{m_C^{(k)}-1}\alpha_{C,i}^{(k)} x_C^{(k-m_C^{(k)}+i)}+\alpha_{C,m_C^{(k)}}^{(k)} \widetilde x_C^{(k)}\\
\downarrow\\
\widetilde x_C^{(k+1)}=x_C^{(k)}+f_C^{(k)}\\
\downarrow\\
\widetilde f_C^{(k+1)}:=f(\widetilde x_C^{(k+1)})\\
\downarrow\\
\mbox{Find } \alpha_C^{(k+1)} \mbox{ that minimizes } \|f_C^{(k+1)}\|_2, \\ \mbox{with } f_C^{(k+1)}=\sum\limits_{i=0}^{m_C^{(k+1)}-1}\alpha_{C,i}^{(k+1)} f_C^{(k+1-m_C^{(k+1)}+i)}+\alpha_{C,m_C^{(k+1)}}^{(k+1)} \widetilde f_C^{(k+1)}\\
\downarrow\\
x_C^{(k+1)}=\sum\limits_{i=0}^{m_C^{(k+1)}-1}\alpha_{C,i}^{(k+1)} x_C^{(k+1-m_C^{(k+1)}+i)}+\alpha_{C,m_C^{(k+1)}}^{(k+1)} \widetilde x_C^{(k+1)}
\end{matrix}
\begin{matrix}
\vphantom{\widetilde x_C^{(k)}=x_C^{(k-1)}+f_C^{(k-1)}}\\
\vphantom{\downarrow}\\
\vphantom{\widetilde f_C^{(k)}=f(\widetilde x_C^{(k)})}\\
\vphantom{\downarrow}\\
\vphantom{\mbox{Find } \alpha_C^{(k)} \mbox{ that minimizes } \|f_C^{(k)}\|_2} \\ \vphantom{\mbox{ with }  f_C^{(k)}=\sum\limits_{i=0}^{m_C^{(k)}-1}\alpha_{C,i}^{(k)} f_C^{(k-m_C^{(k)}+i)}+\alpha_{C,m_C^{(k)}}^{(k)} \widetilde f_C^{(k)}}\\
\vphantom{\downarrow}\\
\coolrightbrace{x_C^{(k)}=\sum\limits_{i=0}^{m_C^{(k)}-1}\alpha_{C,i}^{(k)} x_C^{(k-m_C^{(k)}+i)}+\alpha_{C,m_C^{(k)}}^{(k)} \widetilde x_C^{(k)}\\
\downarrow\\
\widetilde x_C^{(k+1)}=x_C^{(k)}+f_C^{(k)}\\
\downarrow\\
\widetilde f_C^{(k+1)}=f(\widetilde x_C^{(k+1)})\\
\downarrow\\
\mbox{Find } \alpha_C^{(k+1)} \mbox{ that minimizes } \|f_C^{(k+1)}\|_2, \\ \mbox{ with } f_C^{(k+1)}=\sum\limits_{i=0}^{m_C^{(k+1)}-1}\alpha_i^{(k+1)} f_C^{(k+1-m_C^{(k+1)}+i)}+\alpha_{m_C^{(k+1)}}^{(k+1)} \widetilde f_C^{(k+1)}\\
\downarrow\\
x_C^{(k+1)}=\sum\limits_{i=0}^{m_C^{(k+1)}-1}\alpha_{C,i}^{(k+1)} x_C^{(k+1-m_C^{(k+1)}+i)}+\alpha_{C,m_C^{(k+1)}}^{(k+1)} \widetilde x_C^{(k+1)}
}{(IV)}
\end{matrix}%
\]
\caption{CROP method in CROP notation.}
\label{fig:CROPvsAA}
\vspace{-0.1in}
\end{figure}

\noindent
Hence, blocks (I - IV) in~\Cref{fig:AAvsCROP,fig:CROPvsAA} can be associated with four different algorithms:
block (I) illustrates Anderson Acceleration steps \eqref{eq:WAvg}--\eqref{eq:Anderson}; block (II) is equivalent to steps \eqref{eq:tildex}--\eqref{eq:ufCROP} of CROP algorithm; block (III) is equivalent to the Anderson Acceleration steps \eqref{eq:WAvg}--\eqref{eq:Anderson} and
block (IV) illustrates CROP algorithm steps \eqref{eq:tildex}--\eqref{eq:ufCROP}. Note that following the notation of \Cref{fig:AAvsCROP,fig:CROPvsAA}
allows us to use the same variables $\displaystyle x^{(k)}$ and $\displaystyle \bar x^{(k)}$ in all four different algorithms (blocks I -- IV) and write them all in the unified framework in terms of previous iterates of Anderson Acceleration method. However, we can also express all quantities of interest in terms of CROP past information. Since CROP algorithm has better behavior while keeping less historical information, we can run Anderson Acceleration executing block (III), which is called CROP generalization of Anderson Acceleration~\cite{EttJ15} or, for the simplicity, CROP-Anderson method.

CROP-Anderson method, denoted as $CA$, follows the steps of \Cref{alg:CROP}, with the only difference of checking the size of residuals (errors) at each iteration step and the final output result being $\displaystyle \widetilde x_C^{(k+1)}$.
\begin{algorithm}[htbp!]
\caption{CROP-Anderson method (of fixed depth $m$)}
\begin{algorithmic} [1]
\Require Initial CROP-Anderson iterate $x^{(0)}$, initial fixed depth $m \geq 1$
\State Compute $f_C^{(0)} = f(x_C^{(0)})$
\For{$k=0,1,2,\ldots$}
	\State Set $\widetilde x_C^{(k+1)}=x_C^{(k)}+f_C^{(k)}$ \ and 
 \ $\widetilde f_C^{(k+1)}=f(\widetilde x_C^{(k+1)})$
        \If {$\widetilde f_C^{(k+1)}<tol$} 
            \State break
        \EndIf
	\State Set truncation parameter $m_C^{(k+1)}$ \Statex \quad \ \texttt{/* e.g. $m_C^{(k+1)} = \min\{k+1,m\}$ */}
	\State Set $F_C^{(k+1)}=[f_C^{(k-m_C^{(k)})},\ldots,f_C^{(k)}, \widetilde f_C^{(k+1)}]$.
	\State Determine mixing coefficients, i.e., $\alpha_C^{(k+1)}=[\alpha_{C,0}^{(k+1)},\ldots,\alpha_{C,m_C^{(k+1)}}^{(k+1)}]^T$ that
 \Statex \quad \; solves Problem \eqref{eq:coefCROP}
	\State Set $x_C^{(k+1)}=\sum\limits_{i=0}^{m_C^{(k+1)}-1}\alpha_{C,i}^{(k+1)}x_C^{(k+1-m_C^{(k+1)}+i)}+\alpha_{C,m_C^{(k+1)}}^{(k+1)}\widetilde x_C^{(k+1)}$
        \State Set $f_C^{(k+1)}=\sum\limits_{i=0}^{m_C^{(k+1)}-1}\alpha_{C,i}^{(k+1)}f_C^{(k+1-m_C^{(k+1)}+i)}+\alpha_{C,m_C^{(k+1)}}^{(k+1)}\widetilde f_C^{(k+1)}$
\EndFor
\Ensure $\widetilde x_C^{(k+1)}$ that solves $f(x)=0$.
\end{algorithmic}
\label{alg:CROPAnd}
\end{algorithm}
Although we can express the steps of CROP-Anderson method using CROP algorithm ($C$) notation,  we can easily use the alternative CROP-Anderson ($CA$) formulation, i.e., we first set $\displaystyle x^{(k)}_{CA} = \widetilde x^{(k)}_{C}$, $\displaystyle f^{(k)}_{CA} = \widetilde f^{(k)}_{C}$, $\displaystyle \bar x^{(k)}_{CA} = x^{(k)}_{C}$ and $\displaystyle \bar f^{(k)}_{CA} = f^{(k)}_{C}$. Then by \Cref{thm:AequivC}, we get the following corollary.
\vspace{0.05in}
\begin{corollary}
    Let us consider applying Anderson Acceleration method ($\displaystyle \beta_k=1$) and CROP-Anderson algorithm to the nonlinear problem $\displaystyle f(x) =0$ with initial values $\displaystyle x_A^{(0)} = x_C^{(0)} = x^{(0)}_{CA}$ and no truncation ($\displaystyle m^{(k)}=k$). Then, for $\displaystyle k=0,1,\ldots$
\[
\displaystyle x^{(k)}_{CA} = x_A^{(k)}, \ f^{(k)}_{CA} = f_A^{(k)}
\quad \mbox{ and } \quad \bar x^{(k)}_{CA} = \bar x_A^{(k)}, \ \bar f^{(k)}_{CA} = \bar f_A^{(k)}.
\]
\end{corollary}
In CROP algorithm, updates $\displaystyle x_C^{(k)}$ have the corresponding residuals $\displaystyle f_C^{(k)}$ associated with the least-squares Problem \ref{eq:coefCROP}. Although $\displaystyle f_C^{(k)}$ are used to terminate the iterations, they are still approximated residuals and are affected by the initial guess $\displaystyle x_C^{(0)}$.
Thus, we call them \textit{control residuals}. Relatively, $\displaystyle r_C^{(k)}=f(x_C^{(k)})$ are the \emph{real residuals} corresponding to iterates $\displaystyle x_C^{(k)}$, which obviously are never explicitly computed. Consequently, when solving nonlinear problems, the control residuals may not estimate the real residuals very well, and may cause a lot of problems, e.g. breakdowns. One remedy is to use CROP-Anderson method. Alternatively, the real residuals can be used instead of the control residuals. By changing line 7 of \Cref{alg:CROP} and the corresponding line in CROP-Anderson method (line 11 in \Cref{alg:CROPAnd}) into $\displaystyle f_C^{(k+1)}=f(x_C^{k+1})$, CROP algorithm and CROP-Anderson method become completely different algorithms. In what follows, we will refer to them as rCROP and rCROP-Anderson. Note that in the case of Anderson Acceleration method, the control residuals and the real residuals are the same, i.e., $\displaystyle f_A^{(k)} = f(x_A^{(k)}) = r_A^{(k)}$.

The diagrams above are summarized in \Cref{tab:AvsC} which illustrates two steps of Anderson Acceleration method and CROP algorithm. The $k^{th}$ step of each method is highlighted in \textbf{bold}. Note that the ``optimization'' steps have different results for the two methods, but the ``average $x$'' and the ``iteration $k$'' steps should have the same value. Analogously, the ``average $f$'' and the ``residual $k$'' steps should be the same. If we consider $\displaystyle \widetilde x_C$ and $\displaystyle \widetilde f_C$ from CROP algorithm as iterates and residuals, respectively, then we get steps of CROP-Anderson method  (shaded part of CROP), which is equivalent to Anderson Acceleration method.

If in CROP and CROP-Anderson algorithms we use $\displaystyle f_C^{(k)}=f(x_C^{(k)})$ instead of $\displaystyle f_C^{(k)}= F_C^{(k)}\alpha_C^{(k)}$ as $\displaystyle k^{th}$ residual, then we get real residual CROP (rCROP) and the corresponding rCROP-Anderson (shaded part of CROP method). The rCROP algorithm is presented along Anderson Acceleration method in \Cref{tab:AvsrC}.

\begin{table}[htbp!]
\footnotesize
    \centering
    \hspace{-0.2in}
    \begin{tabular}{|c|c||c|c|}
    \hline
         & \textbf{Anderson Acceleration} & \textbf{CROP Algorithm}& \\ \hline \hline\\[-0.8em]
        iteration $k$ & $x_A^{(k)}=\bar x_A^{(k-1)}+\bar f_A^{(k-1)}$ & $\widetilde x_C^{(k)}=x_C^{(k-1)}+f_C^{(k-1)}$ & new direction \\[0.5em] \hline \\[-0.8em]
        residual $k$ & $f_A^{(k)}=f(x_A^{(k)})$ &  $\widetilde f_C^{(k)}=f(\widetilde x_C^{(k)})$ & new direction \\[0.5em] \hline \\[-0.8em]
        \rowcolor{Gray}
        \textbf{optimization} & $\mathbf{\alpha_A^{(k)}}=\mathbf{\arg\min\|F_A^{(k)}\alpha\|_2}$ & $\alpha_C^{(k)}=\arg\min\|F_C^{(k)}\alpha\|_2$ & optimization\\[0.5em] \hline \\[-0.8em]
        \rowcolor{Gray}
        \textbf{average $x$} & 
        $\mathbf{\bar x_A^{(k)}= X_A^{(k)}\alpha_A^{(k)}}$ & $x_C^{(k)}= X_C^{(k)}\alpha_C^{(k)}$ & iteration $k$  \\[0.5em] \hline \\[-0.8em]
        \rowcolor{Gray}
        \textbf{average $f$} & 
        $\mathbf{\bar f_A^{(k)}= F_A^{(k)}\alpha_A^{(k)}}$ & $f_C^{(k)}=F_C^{(k)}\alpha_C^{(k)}$ & residual $k$\\[0.5em] \hline \\[-0.8em]
        \rowcolor{Gray}
        \textbf{iteration $k+1$} & $\mathbf{x_A^{(k+1)}=\bar x_A^{(k)}+\bar f_A^{(k)}}$ & $\mathbf{\widetilde x_C^{(k+1)}=x_C^{(k)}+f_C^{(k)}}$ & \textbf{new direction}\\[0.5em] \hline \\[-0.8em]
        \rowcolor{Gray}
        \textbf{residual $k+1$} & $\mathbf{f_A^{(k+1)}=f(x_A^{(k+1)})}$ &  $\mathbf{\widetilde f_C^{(k+1)}=f(\widetilde x_C^{(k+1)})}$ & \textbf{new direction} \\[0.5em] \hline \\[-0.8em]
        optimization & \hspace{-0.06in}$\alpha_A^{(k+1)}=\arg\min\|F_A^{(k+1)}\alpha\|_2$\hspace{-0.06in} & \hspace{-0.06in} $\mathbf{\alpha_C^{(k+1)}=\arg\min\|F_C^{(k+1)}\alpha\|_2}$\hspace{-0.06in} & \textbf{optimization}\\[0.5em] \hline \\[-0.8em]
        average $x$ & $\bar x_A^{(k+1)}= X_A^{(k+1)}\alpha_A^{(k+1)}$ & $\mathbf{x_C^{(k+1)}= X_C^{(k+1)}\alpha_C^{(k+1)}}$ & \textbf{iteration $k+1$}  \\[0.5em] \hline \\[-0.8em]
        average $f$ & $\bar f_A^{(k+1)}= F_A^{(k+1)}\alpha_A^{(k+1)}$ & $\mathbf{f_C^{(k+1)}=F_C^{(k+1)}\alpha_C^{(k+1)}}$ & \textbf{residual $k+1$}\\[0.5em] \hline
    \end{tabular}
    \caption{Two iteration steps of Anderson Acceleration and CROP method.}
    \label{tab:AvsC}
\end{table}

\begin{table}[htbp!]
\footnotesize
    \centering
    \hspace{-0.2in}
    \begin{tabular}{|c|c||c|c|}
    \hline
         & \textbf{Anderson Acceleration} & \textbf{\textcolor{red}{rCROP} Algorithm} & \\ \hline \hline\\[-0.8em]
        iteration $k$ & $x_A^{(k)}=\bar x_A^{(k-1)}+\bar f_A^{(k-1)}$ & $\widetilde x_C^{(k)}=x_C^{(k-1)}+f_C^{(k-1)}$ & new direction\\[0.5em] \hline \\[-0.8em]
        residual $k$ & $f_A^{(k)}=f(x_A^{(k)})$ &  $\widetilde f_C^{(k)}=f(\widetilde x_C^{(k)})$ & new direction \\[0.5em] \hline \\[-0.8em]
        \rowcolor{Gray}
        \textbf{optimization} & $\mathbf{\alpha_A^{(k)}=\arg\min\|F_A^{(k)}\alpha\|_2}$ & $\alpha_C^{(k)}=\arg\min\|F_C^{(k)}\alpha\|_2$ & optimization\\[0.5em] \hline \\[-0.8em]
        \rowcolor{Gray}
        \textbf{average $x$} & $\mathbf{\bar x_A^{(k)}= X_A^{(k)}\alpha_A^{(k)}}$ & $x_C^{(k)}= X_C^{(k)}\alpha_C^{(k)}$ & iteration $k$  \\[0.5em] \hline \\[-0.8em]
        \rowcolor{Gray}
        \textbf{average $f$} & $\mathbf{\bar f_A^{(k)}= F_A^{(k)}\alpha_A^{(k)}}$ & \textcolor{red}{$\mathbf{f_C^{(k)}=f(x_C^{(k)})}$} & residual $k$\\\hline
        \rowcolor{Gray}
      \textbf{iteration $k+1$} & $\mathbf{x_A^{(k+1)}=\bar x_A^{(k)}+\bar f_A^{(k)}}$ & $\mathbf{\widetilde x_C^{(k+1)}=x_C^{(k)}+f_C^{(k)}}$ & \textbf{new direction}\\[0.5em] \hline \\[-0.8em]
        \rowcolor{Gray}
        \textbf{residual $k+1$} & $\mathbf{f_A^{(k+1)}=f(x_A^{(k+1)})}$ &  $\mathbf{\widetilde f_C^{(k+1)}=f(\widetilde x_C^{(k+1)})}$ & \textbf{new direction} \\[0.5em] \hline \\[-0.8em]
        optimization & \hspace{-0.06in}$\alpha_A^{(k+1)}=\arg\min\|F_A^{(k+1)}\alpha\|_2$\hspace{-0.06in} & \hspace{-0.06in}$\mathbf{\alpha_C^{(k+1)}=\arg\min\|F_C^{(k+1)}\alpha\|_2}$\hspace{-0.06in} & \textbf{optimization}\\[0.5em] \hline \\[-0.8em]
        average $x$ & $\bar x_A^{(k+1)}= X_A^{(k+1)}\alpha_A^{(k+1)}$ & $\mathbf{x_C^{(k+1)}= X_C^{(k+1)}\alpha_C^{(k+1)}}$ & \textbf{iteration $k+1$}  \\[0.5em] \hline \\[-0.8em]
        average $f$ & $\bar f_A^{(k+1)}= F_A^{(k+1)}\alpha_A^{(k+1)}$ & \textcolor{red}{$\mathbf{f_C^{(k+1)}=f(x_C^{(k+1)})}$} & \textbf{residual $k+1$}\\[0.5em] \hline 
    \end{tabular}
    \caption{Two iteration steps of Anderson Acceleration and \textcolor{red}{\bf rCROP} method.}
    \label{tab:AvsrC}
\end{table}


\section{Broader View of Anderson Acceleration Method and CROP Algorithm}

This section discusses the connections between CROP algorithm and some other state-of-the-art iterative methods. Following our findings in \Cref{sec:AndersonVsCROP}, we explore links between CROP and Krylov subspace methods in \Cref{sec:ConKrylov}, and multisecant methods in \Cref{sec:ConMulti}.

In \cite{Washio1997, Oos2000}, a Krylov acceleration method equivalent to flexible GMRES was introduced, and its Jacobian-free version utilizes the least-squares similar to \Cref{eq:constainedLSA}. \cite{RohS11} pointed out the connection between Anderson Acceleration and the GMRES method, and details on the equivalence between Anderson Acceleration without truncation and the GMRES method were provided~\cite{Ni09,WalN11}.  \cite{Hans2024} showed that truncated Anderson Acceleration is a multi-Krylov method.

\subsection{Connection with Krylov methods}
\label{sec:ConKrylov}

Let us consider applying CROP algorithm to the simple linear problem: Find $x \in \R^n$ such that $\displaystyle Ax=b$, with a nonsingular $\displaystyle A \in \R^{n \times n}$ and $\displaystyle b \in \R^n$. Then, the associated residual (error) vectors can be chosen as $\displaystyle f(x) = b - Ax$, with the corresponding $\displaystyle g(x) = b + (I-A)x$.

First, we will present some facts about CROP algorithm's application to the linear problem $\displaystyle  Ax = b$.

\begin{lemma}\label{lem:linres}
    Consider using CROP algorithm to solve the linear problem $Ax=b$. Then, the real and the control residuals are equal, i.e., $\displaystyle r_C^{(k)} = f_C^{(k)}$ for any $\displaystyle  k\in\N$.
\end{lemma}
\begin{proof}
    The proof follows by induction. For the initial residuals we have $\displaystyle r_C^{(0)}=f_C^{(0)}$. Assume that $\displaystyle r_C^{(\ell)}=f_C^{(\ell)}$ for all $\displaystyle \ell \le k$. Then, for $k+1$
\begin{equation}
\nonumber
\begin{split}
\displaystyle
f_C^{(k+1)}&=\sum_{i=0}^{m_C^{(k+1)}-1}\alpha_{C,i}^{(k+1)} f_C^{(k+1-m_C^{(k+1)}+i)}+\alpha_{C,m_C^{(k+1)}}^{(k+1)} \widetilde f_C^{(k+1)}\\
&=\sum_{i=0}^{m_C^{(k+1)}-1}\alpha_{C,i}^{(k+1)} r_C^{(k+1-m_C^{(k+1)}+i)}+\alpha_{C,m_C^{(k+1)}}^{(k+1)} \widetilde f_C^{(k+1)}\\
&=\sum_{i=0}^{m_C^{(k+1)}-1}\alpha_{C,i}^{(k+1)}\big(b-Ax_C^{(k+1-m_C^{(k+1)}+i)}\big)+\alpha_{C,m_C^{(k+1)}}^{(k+1)} \big(b-A\widetilde x_C^{(k+1)}\big)\\
    &=b-A\bigg(\sum_{i=0}^{m_C^{(k+1)}-1}\alpha_{C,i}^{(k+1)} x_C^{(k+1-m_C^{(k+1)}+i)}+\alpha_{C,m_C^{(k+1)}}^{(k+1)} \widetilde x_C^{(k+1)}\bigg)\\
    &=b-Ax_C^{(k+1)} = r_C^{(k+1)}.
\end{split}
\end{equation}
Hence, by induction $\displaystyle r_C^{(k)}=f_C^{(k)}$ \ for all $k\in\N$.
\end{proof}

\begin{lemma}\label{lem:linrestilde}
    If CROP algorithm is used to solve the linear problem $\displaystyle Ax=b$, then $\displaystyle \widetilde f_C^{(k+1)} = (I-A) f_C^{(k)}$ for any $k\in\N$.
\end{lemma}

\begin{proof}
\begin{equation}
\nonumber
\begin{split}
    \widetilde f_C^{(k+1)} &= b-A\widetilde x_C^{(k+1)} = b-A\Big(b+(I-A) x_C^{(k)}\Big)=(I-A)(b-A x_C^{(k)})\\
    & =(I-A) r_C^{(k)}=(I-A) f_C^{(k)}.
\end{split}
\end{equation}
\end{proof}

The equivalence of Anderson Acceleration method and the GMRES method is presented in \cite{Ni09,WalN11}. A similar result for CROP algorithm can be proved.

\begin{theorem}\label{thm:CeqivG}
    Consider using CROP algorithm and the GMRES method to solve the linear problem $\displaystyle Ax=b$ under the following assumptions:
    \begin{enumerate}
        \item $A$ is nonsingular.
        \item Run CROP algorithm with $\displaystyle f(x)=b-Ax$, $\displaystyle g(x)=f(x)+x=b+(I-A)x$ and no truncation.
        \item The initial values of the GMRES and CROP coincide, i.e., $\displaystyle x_G^{(0)}=x_C^{(0)}$.
    \end{enumerate}
    Then, for $\displaystyle  k>0$, $\displaystyle  x_C^{(k)}=x_G^{(k)}$, $\displaystyle  f_C^{(k)}=r_G^{(k)}$, where $r_G^{(k)}$ denote the GMRES residual defined as $\displaystyle r_G^{(k)} := b-Ax_G^{(k)}$.
\end{theorem}

Before proving \Cref{thm:CeqivG}, we first need the following lemma.
\begin{lemma}\label{lem:CROPKrylov}
    Consider using CROP algorithm and the GMRES method to solve the linear problem $\displaystyle Ax=b$ under the assumptions from \Cref{thm:CeqivG}, and let the Krylov subspace associated with matrix $\displaystyle A$ and vector $\displaystyle r_G^{(0)}$ be given as
    $$\displaystyle \mathcal{K}_n=\mathcal{K}_n(A,r_G^{(0)})=span\{r_G^{(0)},Ar_G^{(0)},\ldots, A^{n-1}r_G^{(0)}\}.$$
    Then, $\displaystyle \mathcal{K}_n=span\{f_C^{(0)},\ldots, f_C^{(n-1)}\}$ for all $\displaystyle n\in\mathbb{N^+}$
\end{lemma}
\begin{proof}
    The proof proceeds by induction. 
    For $\displaystyle n=1$, the Krylov subspace $\displaystyle \mathcal{K}_1=span\{r_G^{(0)}\}=span\{f_C^{(0)}\}$.\\
    Assume that for $\displaystyle n=k+1$, the Krylov subspace $\displaystyle \mathcal{K}_{k+1}=span\{f_C^{(0)},\ldots, f_C^{(k)}\}$. Then, for $\displaystyle n=k+2$,
    $$\displaystyle \widetilde f_C^{(k+1)}=(I-A) f_C^{(k)}=f_C^{(k)}-A f_C^{(k)}\in \mathcal{K}_{k+2}$$
    Since $\displaystyle f_C^{(k+1)}$ is a linear combination of vectors $\displaystyle f_C^{(0)},\ldots, f_C^{(k)},\widetilde f_C^{(k+1)}$, thus $\displaystyle f_C^{(k+1)}\in \mathcal{K}_{k+2}$. 
\end{proof}
Finally, we are ready to prove~\Cref{thm:CeqivG}.
\begin{proof}[Proof of \Cref{thm:CeqivG}]

    We show by induction that at step $k$ of CROP algorithm vectors $\displaystyle f_C^{(0)},\ldots, f_C^{(k)}$ form 
     Krylov subspace $\displaystyle \mathcal{K}_{k+1}$, i.e., $\displaystyle \mathcal{K}_{k+1} :=span\big\{f_C^{(0)},\ldots, f_C^{(k)}\big\}$.
    \noindent
    For $k=0$, Krylov subspace $\mathcal{K}_1=span\big\{r_G^{(0)}\big\}=span\big\{f_C^{(0)}\big\}$, which proves the basic case.\\
    Assume that at step $k$ we have Krylov subspace $\displaystyle \mathcal{K}_{k+1}=span\big\{f_C^{(0)},\ldots, f_C^{(k)}\big\}$ as an induction hypothesis. Then at step $k+1$
    $$\displaystyle \widetilde f_C^{(k+1)}=(I-A) f_C^{(k)}=f_C^{(k)}-Af_C^{(k)}\in \mathcal{K}_{k+2}.$$
    Since $f_C^{(k+1)}$ is a linear combination of vectors $f_C^{(0)},\ldots, f_C^{(k)},\widetilde f_C^{(k+1)}$,  \
    $ \displaystyle f_C^{(k+1)} \in \mathcal{K}_{k+2}$. Moreover, we need to show that $\displaystyle f_C^{(k+1)}\notin \mathcal{K}_{k+1}$ which requires $\displaystyle \alpha_{C,m_C^{(k+1)}}^{(k+1)}\neq0$ and $\displaystyle \widetilde f_C^{(k+1)}\notin \mathcal{K}_{k+1}$. These, however, are satisfied unless the algorithm stagnates. The equivalence of CROP algorithm and GMRES method follows directly from the fact that $\displaystyle f_C^{(k)}=\min\limits_{v \in \mathcal{K}_{k+1}} v =r_G^{(k)}$.
\end{proof}

\begin{remark}
    Note that \Cref{thm:CeqivG} can also be proved using \Cref{thm:AequivC} and the equivalence between Anderson Acceleration and GMRES method~\cite{WalN11}.
\end{remark}

Next, we will show that, similarly to full CROP algorithm, there exists a truncated linear method equivalent to the truncated CROP(m) algorithm. In \Cref{thm:CeqivG}, we have shown the equivalence between CROP and GMRES method. Since GMRES is equivalent to Generalized Conjugate Residual (GCR) algorithm, let us consider the truncated GCR, namely, the ORTHOMIN method~\cite[Section 6.9]{Saad2003}). For general CROP($m$) method, the following theorem holds.
\begin{theorem}
\label{thm:CeqivO}
Consider solving the linear system $Ax = b$ by CROP($m$) algorithm and ORTHOMIN($m-1$) method, under the following assumptions:
    \begin{enumerate}
        \item $A$ is nonsingular.
        \item Run CROP($m$) algorithm with $\displaystyle f(x)=b-Ax$ and  $\displaystyle g(x)=f(x)+x=b+(I-A)x$.
        \item CROP algorithm is truncated with parameter $m$, and ORTHOMIN with parameter $m-1$.
        \item The initial values of CROP($m$) and ORTHOMIN($m-1$) coincide, i.e.,\\ $\displaystyle x_{Omin(m-1)}^{(0)} = x_{CROP(m)}^{(0)}$.\\[0.01in]
    \end{enumerate}
    Then for $\displaystyle k>0$, \ $\displaystyle x_{CROP(m)}^{(k)}=x_{Omin(m-1)}^{(k)}$ \ and \ $\displaystyle f_{CROP(m)}^{(k)}=r_{Omin(m-1)}^{(k)}$, with ORTHOMIN residuals \ $\displaystyle r_{Omin(m-1)}^{(k)}=b-Ax^{(k)}$.
\end{theorem}
\begin{proof}
    We proof \Cref{thm:CeqivO} by induction on $m$.
    
    Let us start with $m = 1$. At step $k+1$ of CROP(1) iterate $\displaystyle x_C^{(k+1)}$ is determined directly from $\displaystyle x_C^{(k)}$ and $\displaystyle f_C^{(k)}$, hence making it a fixed-point iteration. According to~\eqref{eq:uffCROP}, the control residuals in CROP(1) are given as
    \begin{equation}
    \label{eq:CROP(1)_control_res}
    \displaystyle 
    f_C^{(k+1)}=\alpha_{C,0}^{(k+1)}f_C^{(k)}+\alpha_{C,1}^{(k+1)}\widetilde f_C^{(k+1)}.
    \end{equation}
    Since $\displaystyle \widetilde f_C^{(k+1)}=(I-A)f_C^{(k)}$ \ and \ $\displaystyle \alpha_{C,0}^{(k+1)}=1-\alpha_{C,1}^{(k+1)}$,  \eqref{eq:CROP(1)_control_res} yields
    \begin{equation}
    \label{eq:CROP1f}
        f_C^{(k+1)}=\big(1-\alpha_{C,1}^{(k+1)}\big)f_C^{(k)}+\alpha_{C,1}^{k+1}\big(I-A\big) f_C^{(k)}=f_C^{(k)}-\alpha_{C,1}^{(k+1)} Af_C^{(k)}.
    \end{equation}
    To obtain a solution of the least-squares problem~\eqref{eq:coefCROP} which minimizes $\displaystyle \|f_C^{(k+1)}\|_2$, 
    \ we need $\displaystyle f_C^{(k+1)}\perp Af_C^{(k)}$ and thus 
    \begin{equation}
    \label{eq:CROP1alpha}
    \alpha_{C,1}^{(k+1)}=\frac{(f_C^{(k)})^TAf_C^{(k)}}{(Af_C^{(k)})^TAf_C^{(k)}}. 
    \end{equation}
    Therefore,
     \begin{equation}
    \label{eq:CROP1x}
    \displaystyle
    x_C^{(k+1)} =\big(1-\alpha_{C,1}^{(k+1)}\big)x_C^{(k)}+\alpha_{C,1}^{(k+1)} \big(x_C^{(k)}+f_C^{(k)}\big)= x_C^{(k)}+\alpha_{C,1}^{(k+1)}f_C^{(k)}.
    \end{equation}
    %
    Hence, iterates $x_C^{(k+1)}$ and residuals $\displaystyle f_C^{(k+1)}$ are updated according to \eqref{eq:CROP1x} and \eqref{eq:CROP1f}, respectively, with $\alpha_{C,1}^{(k+1)}$ computed as in \eqref{eq:CROP1alpha}. Moreover, they are exactly the same as those calculated by the ORTHOMIN($0$), i.e., CROP($1$) = ORTHOMIN($0$).

    Assume that CROP($\ell$) = ORTHOMIN($\ell-1$) for \ $\ell=1,\ldots,m-1$. Then, we can prove CROP($m$) = ORTHOMIN($m-1$) by induction on $k$.
    
    Since $\displaystyle m_C^{(k+1)} = \min\{k+1,m\}$, for $\displaystyle k=0$, $\displaystyle m_C^{(k+1)} = 1$ and CROP($m$) = CROP($1$) = ORTHOMIN($0$), which is the same as the first step of ORTHOMIN($m$) for any $\displaystyle m > 1$.
    For $\displaystyle k+1 < m$, step $\displaystyle k+1$ of CROP($m$) is the same as step $\displaystyle k+1$ of CROP($k+1$), which is equivalent to step $\displaystyle k+1$ of ORTHOMIN($k$) by the induction assumption, and thus is the same as step $\displaystyle k+1$ of ORTHOMIN($m-1$). Hence, we only need to consider the case of $k+1\ge m$. 
    
    Assume that for some $\displaystyle k > 0$ vectors $\Big\{A\Delta x_C^{(k+1-m_C^{(k)})},\ldots, A\Delta x_C^{(k-1)},f_C^{(k)}\Big\}$ are pairwise orthogonal and $\displaystyle x_{CROP(m)}^{(k)} = x_{Omin(m-1)}^{(k)}$. Then, for general CROP(m), the $k+1$ residual is 
    \begin{equation}
        \displaystyle
        f_C^{(k+1)}=\sum\limits_{i=0}^{m-1}\alpha_{C,i}^{(k+1)}f_C^{(k+1-m)}+\alpha_{C,m}^{(k+1)}\widetilde f_C^{(k+1)}.
    \end{equation}
    By \Cref{lem:linrestilde} \ $\displaystyle \widetilde f_C^{(k+1)}=(I-A)f_C^{(k)}$. Then, with  $\displaystyle \alpha_{C,0}^{(k+1)}=\gamma_{C,1}^{(k+1)}$, $\displaystyle \alpha_{C,i}^{(k+1)}=\gamma_{C,i+1}^{(k+1)}-\gamma_{C,i}^{(k+1)}$ and $\displaystyle \gamma_{C,m}^{(k+1)}=1-\alpha_{C,m}^{(k+1)}$, and noting that $\displaystyle \Delta f_C^{(k-m+i)}=-A\Delta x_C^{(k-m+i)}$, we get
\begin{equation}\label{eq:CROPmf}
\begin{split}
\displaystyle
    f_C^{(k+1)}&=\sum\limits_{i=0}^{m}\alpha_{C,i}^{(k+1)}f_C^{(k+1-m)}+\alpha_{C,m}^{(k+1)}(I-A) f_C^{(k)}\\
    &=f_C^{(k)}-\alpha_{C,m}^{k+1} Af_C^{(k)}-\sum_{i=1}^{m-1}\gamma_{C,i}^{(k+1)}A\Delta x_C^{(k-m+i)}.
\end{split}
\end{equation}
Since by the induction assumption $\displaystyle \{A\Delta x_C^{(k-m+1)},\ldots,A\Delta x_C^{(k-1)},f_C^{(k)}\}$ are pairwise orthogonal, if we orthogonalize $\displaystyle Af_C^{(k)}$ against all $\displaystyle A\Delta x_C^{(k-m+i)}$ \ for \ $i=1,\ldots,m-1$ and let
\begin{equation}\label{eq:CROPortho}
\displaystyle
    Ap^{(k)}=Af_C^{(k)}-\sum_{i=1}^{m-1} \frac{(Af_C^{(k)})^TA\Delta x_C^{(k-m+i)}}{(A\Delta x_C^{(k-m+i)})^TA\Delta x_C^{(k-m+i)}} A\Delta x_C^{(k-m+i)},
\end{equation}
then~\eqref{eq:CROPmf} yields
$
\displaystyle
    f_C^{(k+1)}=f_C^{(k)}-\alpha_{C,m}^{k+1} Ap^{(k)}+\sum_{i=1}^{m-1}c_iA\Delta x_C^{(k-m+i)}.$
To obtain a solution of the least-squares problem \eqref{eq:coefCROP} which minimizes $\displaystyle \|f_C^{(k+1)}\|_2$, \ we need
$\displaystyle f_C^{(k+1)}\perp Ap^{(k)}$ and $\displaystyle f_C^{(k+1)}\perp A\Delta x_C^{(k-m+i)}$ for all $\displaystyle i=1,\ldots,m-1$. Thus 
\begin{equation}\label{eq:CROPmalpha}
\displaystyle
  \alpha_{C,m}^{(k+1)}=\frac{(f_C^{(k)})^TAp^{(k)}}{(Ap^{(k)})^TAp^{(k)}}  
\end{equation}
and $c_i=0$ \ for all \ $i=1,\ldots,m-1$. The actual updates of $\displaystyle x_C^{(k+1)}$ and $\displaystyle f_C^{(k+1)}$ are
\begin{equation}
\label{eq:CROPmxf}
\displaystyle
    x_C^{(k+1)} =x_C^{(k)}+\alpha_{C,m}^{(k+1)} p^{(k)} \qquad \mbox{ and } \qquad
    f_C^{(k+1)} =f_C^{(k)}-\alpha_{C,m}^{(k+1)} Ap^{(k)}.
\end{equation}
Since $\displaystyle \Delta x_C^{(k)}=\alpha_{C,m}^{(k+1)} p^{(k)}$ \ is parallel to $p^{(k)}$, we can rewrite \eqref{eq:CROPortho} by replacing $\displaystyle \Delta x_C^{(k-m+i)}$ \ with \ $p^{(k-m+i)}$. Hence, the orthogonalization step \eqref{eq:CROPortho} and the update formulas \eqref{eq:CROPmxf} are the same as those in ORTHOMIN($m-1$) method. Therefore, in the case of solving linear system, the CROP($m$) algorithm is equivalent to ORTHOMIN($m-1$).
\end{proof}

\begin{remark}
    Since for symmetric matrix $A$, ORTHOMIN(0) is Minimal Residual (MINRES) method~\cite[Section 5.3.2]{Saad2003}, and ORTHOMIN(1) is the CR method, for symmetric linear systems CROP(1) is equivalent to MINRES and CROP(2) to the CR method.
\end{remark}

\subsection{Connection with Multisecant Methods}
\label{sec:ConMulti}
Generalized Broyden's second method is a multisecant method~\cite{FanS09} equivalent to Anderson Acceleration method~\cite{Eye96,FanS09} with iterates generated according to the update formula
\begin{equation}\label{eq:guf}
\displaystyle
    x^{(k+1)}=x^{(k)}-G^{(k)}f^{(k)},
\end{equation}
where $\displaystyle G^{(k)}$ is approximated inverse of the Jacobian updated by 
\begin{equation}\label{eq:gufxmJ}
\displaystyle
    G^{(k)}= G^{(k-m)}+(\mathscr{X}^{(k)}-G^{(k-m)}\mathscr{F}^{(k)})\Big[(\mathscr{F}^{(k)})^T\mathscr{F}^{(k)}\Big]^{-1}(\mathscr{F}^{(k)})^T,
\end{equation}
which minimizes $\displaystyle \|G^{(k)}-G^{(k-m)}\|_F$ \ subject to \
$
\displaystyle
G^{(k)}\mathscr{F}^{(k)}=\mathscr{X}^{(k)}.
$
Here, $\displaystyle \mathscr{X}^{(k)}$ and $\displaystyle \mathscr{F}^{(k)}$ represent the differences of iterates and residuals, respectively, i.e.,
$$
\displaystyle
\mathscr{X}^{(k)}=[\Delta x^{(k-m)},\ldots,\Delta x^{(k-1)}]
\quad \mbox{ and } \quad
\mathscr{F}^{(k)}=[\Delta f^{(k-m)},\ldots,\Delta f^{(k-1)}].
$$
Therefore, the update formula \eqref{eq:guf} can be written as

\begin{equation}\label{eq:gufxm}
\displaystyle
    x^{(k+1)}=x^{(k)}-G^{(k-m)}f^{(k)}-(\mathscr{X}^{(k)}-G^{(k-m)}\mathscr{F}^{(k)})\Big[(\mathscr{F}^{(k)})^T\mathscr{F}^{(k)}\Big]^{-1}(\mathscr{F}^{(k)})^Tf^{(k)}.
\end{equation}
When $\displaystyle G^{(k-m)}=-\beta I$, $\displaystyle \mathscr{X}^{(k)}=\mathscr{X}_A^{(k)}$ and $\displaystyle \mathscr{F}^{(k)}=\mathscr{F}_A^{(k)}$, \eqref{eq:gufxm} reduces to \eqref{eq:aufxm}.
Anderson Acceleration forms an approximate inverse of the Jacobian implicitly
$$\displaystyle
G_A^{(k)}=-\beta I+(\mathscr{X}^{(k)}+\beta\mathscr{F}_A^{(k)})\Big[(\mathscr{F}_A^{(k)})^T\mathscr{F}_A^{(k)}\Big]^{-1}(\mathscr{F}_A^{(k)})^T,
$$ 
that minimizes $\|G_A^{(k)}+\beta I\|_F$ subject to 
$
\displaystyle
G_A^{(k)}\mathscr{F}_A^{(k)}=\mathscr{X}_A^{(k)}.
$
Following the same notation, the updates of CROP algorithm can be written as
\begin{equation}
\label{eq:cropuf}
\displaystyle
 x_C^{(k+1)} = \widetilde x_C^{(k+1)} -\mathscr{X}_C^{(k+1)}\Big[(\mathscr{F}_C^{(k+1)})^T\mathscr{F}_C^{(k+1)}\Big]^{-1}(\mathscr{F}_C^{(k+1)})^T\widetilde f_C^{(k+1)}.
\end{equation}
With a common framework in place, we can now
describe the connection between CROP algorithm and multisecant methods. First, if
we consider $\displaystyle \widetilde x_C^{(k)}$ as an iterate, then we can view \eqref{eq:cropuf} as an update of a generalized Broyden's second method.
%
%
%
When $\displaystyle G^{(k-m)}=0$, $\displaystyle \mathscr{X}^{(k)}=\mathscr{X}_C^{(k)}$ and  $\displaystyle \mathscr{F}^{(k)}=\mathscr{F}_C^{(k)}$, \eqref{eq:gufxm} reduces to \eqref{eq:cropuf}.
CROP algorithm forms implicitly an approximate inverse of the Jacobian 
$$
\displaystyle
G_C^{(k+1)}=\mathscr{X}_C^{(k+1)}\Big[(\mathscr{F}_C^{(k+1)})^T\mathscr{F}_C^{(k+1)}\Big]^{-1}(\mathscr{F}_C^{(k+1)})^T,
$$ 
that minimizes $\displaystyle \|G_C^{(k+1)}\|_F$ subject to $\displaystyle G_C^{(k+1)}\mathscr{F}_C^{(k+1)}=\mathscr{X}_C^{(k+1)}$.
\section{Convergence Theory for CROP Algorithm}
\label{sec:convergence}

In this section, we provide some initial results on the convergence of CROP algorithm. 
The first convergence result for Anderson Acceleration method was given in~\cite{TotK15}, under the assumption of a contraction mapping. Following this work, several other convergence results utilizing various different assumptions were established, see for example 
\cite{ChuDE20,EvaPRX20,Pol21,PolRX19,Reb23}. Most of the existing assumptions are needed to determine the existence and uniqueness of the exact solution $x^*$ of $f(x) = 0$ in an open set. Mappings $f$ and $g$ are usually chosen to be Lipschitz continuous. In this section, we use the same assumptions and follow the same process as the one in~\cite{TotK15}.
We first prove in \Cref{subsec:linear} that for the linear case CROP algorithm is q-linearly convergent. Then, in \Cref{subsec:nonlinear}, we show that for the general nonlinear case, the convergence is r-linear.

\subsection{Convergence of CROP Algorithm for Linear Problems}
\label{subsec:linear}

Let us consider applying the CROP algorithm to the simple linear problem $\displaystyle Ax=b$, with a nonsingular matrix $\displaystyle A \in \R^{n \times n}$ and vector $\displaystyle b \in \R^n$. Then, the associated residual (error) vectors can be chosen as $\displaystyle f(x) = b - Ax$, with the corresponding $\displaystyle g(x) = b + (I-A)x$.

Let us present our first convergence result.

\begin{theorem}\label{thm:CROPlin}
Let us consider solving the linear system $\displaystyle Ax = b$. If $\displaystyle \|I-A\|=c<1$, then CROP algorithm converges to the exact solution $\displaystyle x^*=A^{-1}b$, and the control and real residuals converge q-linearly to zero with the q-factor $c$.
\end{theorem}
\begin{proof} Since 
$ \displaystyle f_C^{(k+1)}=\sum_{i=0}^{m_C^{(k+1)}-1}\alpha_{C,i}^{(k+1)} f_C^{(k+1-m_C^{(k+1)}+i)}+\alpha_{C,m_C^{(k+1)}}^{(k+1)} \widetilde f_C^{(k+1)}$
is the least-squares residual corresponding to \eqref{eq:coefCROP}, we have 
$$\displaystyle \|f_C^{(k+1)}\|_2 \le\|\widetilde f_C^{(k+1)}\|_2.$$
Moreover, by \Cref{lem:linrestilde}
$\displaystyle \widetilde f_C^{(k+1)} =(I-A) f_C^{(k)}.$
Finally by \Cref{lem:linres}, we obtain
\begin{equation}
\displaystyle
\nonumber
\|f_C^{(k+1)}\|_2 \le\|\widetilde f_C^{(k+1)}\|_2\le c\|f_C^{(k)}\|_2 \quad \mbox{ and } \quad \|r_C^{(k+1)}\|_2 \le c\|r_C^{(k)}\|_2.
\end{equation} 
\end{proof}

Now, if we set $\displaystyle e^{(k)} = x^{(k)}-x^*$, then $\displaystyle f(x^{(k)})=b-Ax^{(k)}=A\big(x^*-x^{(k)}\big)=-Ae^{(k)}$. 
Consequently, following \Cref{thm:CROPlin} yields
$$
\displaystyle
(1-c)\|e^{(k)}\|_2\le\|f(x^{(k)})\|_2\le c^k\|f(x^{(0)})\|_2\le c^k(1+c)\|e^{(0)}\|_2.$$
Hence, $$\|e^{(k)}\|_2\le\left(\frac{1+c}{1-c}\right)c^k\|e^{(0)}\|_2.$$
To fully understand the convergence of CROP method for linear problems, we refer to the equivalence of CROP and GMRES, and CROP($m$) and ORTHOMIN($m-1$).

\begin{remark}
The convergence of ORTHOMIN is shown in~\cite{Eisenstat1983}. Note that CROP algorithm with no truncation, ORTHOMIN and GMRES are all mathematically equivalent methods. 
\end{remark}


\subsection{Convergence of the CROP Algorithm for Nonlinear Problems}
\label{subsec:nonlinear}

In the case of nonlinear problems, investigating the convergence of fixed-depth CROP~\Cref{alg:CROP} or CROP($m$) algorithm, requires an assumption of the functions $f$ and $g$ to be good enough. Let us consider the following assumption.
\begin{assumption} \label{as:nl} Consider a nonlinear problem \eqref{eq:MainProblem} such that
    \begin{enumerate}
        \item there exists an $x^*$ such that $f(x^*)=0$ and $g(x^*)=x^*$,
        \item function $g$ is Lipschitz continuously differentiable in the ball $\mathcal{B}_{\widehat\rho}(x^*)$ for some $\widehat\rho>0$ with Lipschitz constant $c\in(0,1)$, and
        \item $f'$ is Lipschitz continuous with a Lipschitz constant $L$.
    \end{enumerate}
\end{assumption}
In the case of no truncation, equivalence between CROP algorithm and Anderson Acceleration method was established in \Cref{thm:AequivC}. For truncated CROP($m$) algorithm, we can still try to explore the relation between the two methods. Let us investigate the connection between $\displaystyle \widetilde x_C$ and $\displaystyle x_C$, $\displaystyle \widetilde f_C$ and $\displaystyle f_C$. Note that when $\displaystyle m\neq k$, $\displaystyle \widetilde x_C$ and $\displaystyle \widetilde f_C$ are not the iterates and residuals of Anderson Acceleration method.

For CROP($m$) method, we have
$$
\displaystyle
X_C^{(k)}=\widetilde X_C^{(k)} A_0 A_1\cdots A_k \quad \mbox{ and } \quad
F_C^{(k)}=\widetilde F_C^{(k)} A_0 A_1\cdots A_k,
$$
with
$$
\displaystyle
X_C^{(k)}=\begin{bmatrix} x_C^{(0)}&\cdots&x_C^{(k)}\end{bmatrix}
\quad \mbox{ and } \quad
F_C^{(k)}=\begin{bmatrix} f_C^{(0)}&\cdots&f_C^{(k)}\end{bmatrix},$$
and $\displaystyle A_i$ is the identity matrix with the $\displaystyle (i+1)$-th column changed to the coefficient vector $\displaystyle \alpha_C^{(i)}=\big[\alpha_{C,0}^{(i)},\ldots,\alpha_{C,m_C^{(i)}}^{(i)}\big]^T$ inserted in range of rows starting at index $(i-m_C^{(i)}+1)$ and ending at index $(i+1)$, i.e.,
\[A_0=I, \ A_1=
\begin{bmatrix} 1&\alpha_{C,0}^{(1)}&0&\cdots\\ 
0 &\alpha_{C,1}^{(1)}&0&\cdots\\
0 &0 &1& \cdots\\
\vdots &\vdots&\vdots &\ddots\end{bmatrix}, \
A_k=
\begin{bmatrix} 1&0&\cdots&0&0\\ 
\vdots  &\ddots &\cdots &\vdots &\vdots\\
0       &0      &\cdots &0      &\alpha_{C,0}^{(k)}\\
\vdots  &\vdots &\vdots &\vdots &\vdots\\
0       &0      &\cdots &1      &\alpha_{C,m_C^{(k)}-1}^{(k)}\\
0       &0      &\cdots &0      &\alpha_{C,m_C^{(k)}}^{(k)}
\end{bmatrix}, k = 0,1, \ldots
\]

\noindent
Since matrices $\displaystyle A_i$ have column sum $1$, $\displaystyle x_C^{(k)}$ and $\displaystyle f_C^{(k)}$, as the last column of $\displaystyle X_C^{(k)}$ and $\displaystyle F_C^{(k)}$, respectively, can be written as
$$
\displaystyle
x_C^{(k)}=\sum_{j=0}^{k} s_j^{(k)} \widetilde x_C^{(k)} \quad \mbox{and} \quad f_C^{(k)}=\sum_{j=0}^{k} s_j^{(k)} \widetilde f_C^{(k)}, \quad \mbox{ with } \sum\limits_{j=0}^{k} s_j^{(k)}=1.$$ 

Now, we impose the following assumption that will allow us to establish the convergence result for CROP($m$) along the lines of \cite[Theorem 2.3]{TotK15}.
\begin{assumption} \label{as:nlm}
For all $\displaystyle k>0$, there exists $M > 0$ such that $\displaystyle \sum\limits_{j=0}^{k} |s_j^{(k)}| < M$.
\end{assumption}

Let us consider \cref{as:nl} again. From the Lipschitz condition for function $\displaystyle g$, $\displaystyle \|g(x)-g(x^*)\|_2\le c\|x-x^*\|_2$ or equivalently $\|f(x)-(x-x^*)\|_2\le c\|x-x^*\|_2$. By the triangle inequality we get
\begin{equation}\label{eq:reserr}
\displaystyle
    (1-c)\|x-x^*\|_2\le \|f(x)\|_2\le (1+c)\|x-x^*\|_2.
\end{equation}
Moreover, by the Lipschitz condition on $\displaystyle f'$, there exists $\displaystyle \rho>0$ sufficiently small such that in the ball  $\displaystyle \mathcal{B}_{\rho}(x^*)$ we have 
\begin{equation}\label{eq:fLip}
    \|f(x)-f'(x^*)(x-x*)\|_2\le \frac{L}{2}\|x-x^*\|_2^2,
\end{equation}
which can be written in terms of $g$, as
\begin{equation}\label{eq:gLip}
\displaystyle
    \|g(x)-g'(x^*)(x-x*)-x^*\|_2\le \frac{L}{2}\|x-x^*\|_2^2.
\end{equation}

\begin{theorem}\label{thm:CROPnlm}
    Let \Cref{as:nl} and \Cref{as:nlm} hold and let $\displaystyle c<\widehat c<1$. Then, if $\displaystyle x_{C,m}^{(0)}$ is sufficiently close to $x^*$, the control residual $\displaystyle f_C^{(k)}$ of CROP$(m)$ and CROP-Anderson($m$) algorithm and the residual $\displaystyle f_{CA}^{(k)}=\widetilde f_C^{(k)}$ satisfy
    $$
    \|f_C^{(k)}\|_2\le \widehat c^k\|f_C^{(0)}\|_2  \quad \mbox{ and } \quad \|f_{CA}^{(k)}\|_2\le \widehat c^k \|f_{CA}^{(0)}\|_2.$$
    Moreover, this implies r-linear convergence of CROP$(m)$ and CROP-Anderson($m$) algorithm with r-factor no greater than $\displaystyle \widehat c$.
\end{theorem}

\begin{proof}
    Since $\displaystyle  \lim\limits_{\rho \to 0} \frac{2(1-c)c\widehat c^k+ML\rho}{2(1-c)-L\rho} = c\widehat c^k<\widehat c^{k+1}$, we can choose $\rho$ small enough such that  $\displaystyle \frac{2(1-c)c\widehat c^k+ML\rho}{2(1-c)-L\rho}\le\widehat c^{k+1}$. Suppose $x_C^{(0)}\in \mathcal{B}_\rho(x^*)$, and $x_C^{(0)}$ is sufficiently close to $x^*$ such that $$\displaystyle \frac{M(c+L\rho/2)}{1-c}\|f(x_C^{(0)})\|_2\le\frac{M(1+c)(c+L\rho/2)}{1-c}\|x_C^{(0)}-x^*\|_2\le \rho.
    $$
    Then, we can prove by induction on $k$ that $\displaystyle \|f(\widetilde x_C^{(k)})\|_2\le \widehat c^k\|f(x_C^{(0)})\|_2$ \ and \ $\widetilde x_C^{(k)}\in \mathcal{B}_\rho(x^*)$ for all $k\ge0$. For $k=0$, the conditions are satisfied because $\widetilde x_C^{(0)} = x_C^{(0)}$.
    
    Now, assume that for all $n\le k$, we have $\displaystyle \|f(\widetilde x_C^{(n)})\|_2\le \widehat c^k\|f(x_C^{(0)})\|_2$ and $\widetilde x_C^{(n)}\in \mathcal{B}_\rho(x^*)$.
    By \eqref{eq:gLip} \
    $\displaystyle
    g(\widetilde x_C^{(n)})=x^*+g'(x^*)(\widetilde x_C^{(n)}-x^*)+\Delta^{(n)},$
    with $\|\Delta^{(n)}\|_2\le \frac{L}{2}\|\widetilde x_C^{(n)}-x^*\|_2$. Then, for $n=k+1$, 
    \vspace{-0.3in}
    \begin{equation}
    \nonumber
    \begin{split}
        \widetilde x_C^{(k+1)}&=x^*+\sum_{j=0}^{k} s_j^{(k)} \big(g'(x^*)(\widetilde x_C^{(j)}-x^*)+\Delta^{(j)}\big)\\
        &=x^*+\sum_{j=0}^{k} s_j^{(k)} g'(x^*)(\widetilde x_C^{(j)}-x^*)+\sum_{j=0}^{k}s_j^{(k)}\Delta^{(j)}.
    \end{split}
    \end{equation}
    Since $\displaystyle \big\|\sum_{j=0}^{k}s_j^{(k)}\Delta^{(j)}\big\|_2=\sum_{j=0}^{m_C^{(k)}}\big|s_j^{(k)}\big| \frac{L}{2} \big\|\widetilde x_C^{(j)}-x^*\big\|_2^2$,
    following \eqref{eq:reserr} we get
    $$
    \displaystyle
    \|\widetilde x_C^{(j)}-x^*\|_2<\frac{1}{1-c}\|f(\widetilde x_C^{(j)})\|_2\le\frac{1}{1-c}\|f(x_C^{(0)})\|_2.$$
    Then $\displaystyle \|\widetilde x_C^{(j)}-x^*\|_2<\rho$ yields
    $
    \|\widetilde x_C^{(j)}-x^*\|_2^2\le\frac{\rho}{1-c}\|f(x_C^{(0)})\|_2.
    $
    Under \Cref{as:nlm}, we have 
    \begin{equation}
    \nonumber
    \begin{split}
        \|\widetilde x_C^{(k+1)}-x^*\|_2&=\big\|\sum_{j=0}^{k} s_j^{(k)} g'(x^*)(\widetilde x_C^{(j)}-x^*)+\sum_{j=0}^{k}s_j^{(k)}\Delta^{(j)}\big\|_2\\ 
        &\le 
        \frac{M(c+L\rho/2)}{1-c}\|f(x_C^{(0)})\|_2\le \rho.
    \end{split}
    \end{equation}
    Thus, $\displaystyle \widetilde x_C^{(k+1)}\in\mathcal{B}_{\rho}(x^*)$ and by \eqref{eq:fLip} \ we obtain \ $
    \displaystyle
    f(\widetilde x_C^{(k+1)})=f'(x^*)(\widetilde x_C^{(k+1)}-x^*)+\Delta^{(k+1)},
    $
    with $\displaystyle \|\Delta^{(k+1)}\|_2\le \frac{L}{2}\|\widetilde x_C^{(k+1)}-x^*\|_2^2\le\frac{L\rho}{2(1-c)}\|f(\widetilde x_C^{(k+1)})\|_2$. Since $\displaystyle f'(x^*)=g'(x^*)-I$ \ and \ $g'(x^*)$ commute,
    \begin{equation}
    \nonumber
    \begin{split}
        f(\widetilde x_C^{(k+1)})&=f'(x^*)\sum_{j=0}^{k} s_j^{(k)} g'(x^*)(\widetilde x_C^{(j)}-x^*)+f'(x^*)\sum_{j=0}^{k}s_j^{(k)}\Delta^{(j)}+\Delta^{(k+1)}\\
        &=g'(x^*)\sum_{j=0}^{k} s_j^{(k)} f'(x^*)(\widetilde x_C^{(j)}-x^*)+f'(x^*)\sum_{j=0}^{k}s_j^{(k)}\Delta^{(j)}+\Delta^{(k+1)}\\
        &=g'(x^*)\sum_{j=0}^{k} s_j^{(k)} (f(\widetilde x_C^{(k)})-\Delta^{(j)})+f'(x^*)\sum_{j=0}^{k}s_j^{(k)}\Delta^{(j)}+\Delta^{(k+1)}\\
        &=g'(x^*)\sum_{j=0}^{k} s_j^{(k)} f(\widetilde x_C^{(j)})-\sum_{j=0}^{k}s_j^{(k)}\Delta^{(j)}+\Delta^{(k+1)}\\
        &=g'(x^*)f_C^{(k+1)}-\sum_{j=0}^{k}s_j^{(k)}\Delta^{(j)}+\Delta^{(k+1)}.
    \end{split}
    \end{equation}
    Now, from $\displaystyle \|f_C^{(k)}\|_2\le\|\widetilde f_C^{(k)}\|_2\le \widehat c^k\|f(x_C^{(0)})\|_2$, we have
    \begin{equation}
    \nonumber
    \begin{split}
        \|f(\widetilde x_C^{(k+1)})\|_2&\le \|g'(x^*)f_C^{(k)}\|+\|\sum_{j=0}^{k}s_j^{(k)}\Delta^{(j)}\|+\|\Delta^{(k+1)}\|_2\\
        &\le c\widehat c^k\|f(x_C^{(0)})\|+\frac{ML\rho}{2(1-c)}\|f(x_C^{(0)})\|_2+\frac{L\rho}{2(1-c)}\|f(\widetilde x_C^{(k+1)})\|_2
    \end{split}
    \end{equation}
    Thus \ $\displaystyle
\|f(\widetilde x_C^{(k+1)})\|_2\le\frac{2(1-c)c\widehat c^k+ML\rho}{2(1-c)-L\rho}\|f(x_C^{(0)})\|_2\le\widehat c^{k+1}\|f(x_C^{(0)})\|_2.$
\end{proof}

Although the control residuals $\displaystyle f_C^{(k)}$ are not the real residuals $\displaystyle r_C^{(k)}$, which may cause a breakdown ($\displaystyle f_C^{(k)}=0$) without finding a reasonable approximation, they do have some good properties.

\begin{theorem}
    The control residuals $\displaystyle f_C^{(k)}$ of CROP($m$) for $m \ge 1$ are nonincreasing.
\end{theorem}
\begin{proof} Since the coefficients $\displaystyle\alpha_{C,i}^{(k+1)}$ in
\begin{equation}
\label{eq:Thm5.8}
\displaystyle
f_C^{(k+1)}=\sum_{i=0}^{m_C^{(k+1)}-1}\alpha_{C,i}^{(k+1)} f_C^{(k+1-m_C^{(k+1)}+i)}+\alpha_{C,m_C^{(k+1)}}^{(k+1)} \widetilde f_C^{(k+1)}
\end{equation}
are chosen to minimize $\displaystyle \|f_C^{(k+1)}\|$, and $f_C^{(k)}$ is itself an element in the sum on the right of \eqref{eq:Thm5.8}, we must have  $\displaystyle \|f_C^{(k+1)}\|\le\|f_C^{(k)}\|$.
\end{proof}
In the case of smaller values of $m$, the control residuals are approximating better the real residuals. Also, using CROP-Anderson or rCROP algorithm is a good way to avoid the breakdown. We will see some examples in \Cref{sec:numerical}.

\section{Numerical Experiments}
\label{sec:numerical}

In this section, we present a small selection of numerical results illustrating some observations discussed in the previous sections. We consider both linear and nonlinear problems of various sizes. Details regarding implementation and some further numerical examples can be found in Supplementary Materials, Section SM.1 and SM.2.


\subsection{Linear Problems}
\label{subsec:linex}
First, we present numerical experiments to demonstrate the behavior of discussed methods when applied to a linear system $Ax = b$ with a nonsingular matrix $A$.

\begin{example}[Synthetic Problem]
\label{exp:Ex1}
\end{example}
Let us consider a nonsingular matrix $A\in \mathbb{R}^{100\times 100}$ given as a tridiagonal matrix with entries $(1,-4,1)$ and a seven-diagonal matrix with entries $(0,0,1,-4,1,1,1)$, and vector $b\in\mathbb{R}^{100}$ with its first entry $1$ and others $0$.
We choose $f(x) = b-Ax$, $g(x) = x+f(x)$, $maxit=100$, $tol=10^{-10}$, and run Anderson Acceleration method, CROP and CROP-Anderson algorithm with initial vector $x_0 = 0$ and different values of parameter $m$. The corresponding results are shown in \Cref{fig:lin}.
\begin{figure}[!ht]
    \centering
    \begin{subfigure}[b]{0.49\textwidth}
        \centering
        \includegraphics[scale=0.39]{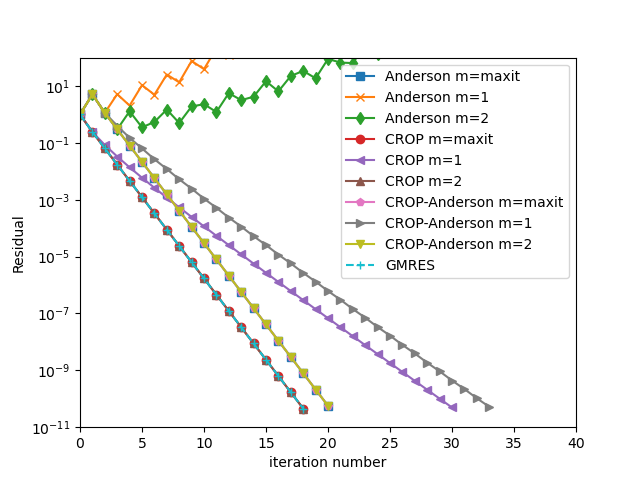}
        \caption{symmetric tridiagonal $A$}
        \label{fig:crop_lin1}
    \end{subfigure}
    \hfill
    \begin{subfigure}[b]{0.49\textwidth}
        \centering
        \includegraphics[scale=0.39]{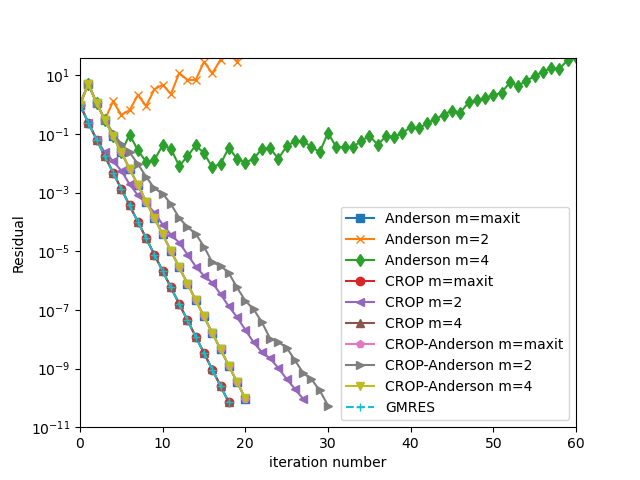}
        \caption{nonsymmetric seven-diagonal $A$}
        \label{fig:crop_lin2}
    \end{subfigure} 
    \caption{A linear problem in Example \eqref{exp:Ex1} with (a) a tridiagonal and (b) a seven-diagonal matrix $A$.}
    \label{fig:lin}
\end{figure}
In \Cref{fig:crop_lin1} we can see that convergence curve of CROP-Anderson method is parallel to the one of CROP algorithm. As we discussed in \Cref{subsec:linear}, CROP algorithm, CROP(2) and GMRES admit the same convergence similarly as Anderson Acceleration method, CROP-Anderson and CROP-Anderson(2) algorithm. Analogous behavior can be observed for CROP, CROP(4) and GMRES method, as well as for Anderson Acceleration method, CROP-Anderson algorithm and CROP-Anderson(4), see \Cref{fig:crop_lin2}.
More linear examples with matrices of different sizes and structures can be found in Supplementary Materials, Section SM.2.





\subsection{Nonlinear Problems}
\label{subsec:nonlinex}
As presented methods are primarily used to accelerate convergence in the case of nonlinear problems, in this section we introduce a variety of nonlinear examples. Starting with small and weakly nonlinear problems, we move towards more complex examples.

\begin{example}[Dominant Linear Part Problem]
\label{exp:Ex4}
\end{example}
Consider a nonlinear problem $\displaystyle Ax+\frac{\mu\|x\|^2}{n}x = b$  with
$\displaystyle  A=tridiag(1,-4,1)\in \mathbb{R}^{n\times n},$ a right-hand side vector $\displaystyle  b\in\mathbb{R}^{n}$ with first entry $1$ and others $0$, $n = 100$ and parameter $\mu=1/100$. We choose $\displaystyle f(x) = Ax+\frac{\mu\|x\|^2}{n}x-b$ \ and \ $\displaystyle  g(x) = x+f(x)$, parameters $\displaystyle  maxit=100$ and $\displaystyle  tol=10^{-10}$, and run Anderson Acceleration method, CROP algorithm and CROP-Anderson method with initial vector $\displaystyle  x^{(0)} = 0$ and different values of parameter $m$. \Cref{fig:nl1} presents our findings.
\begin{figure}[h]
    \centering
    \includegraphics[width=0.49\textwidth]{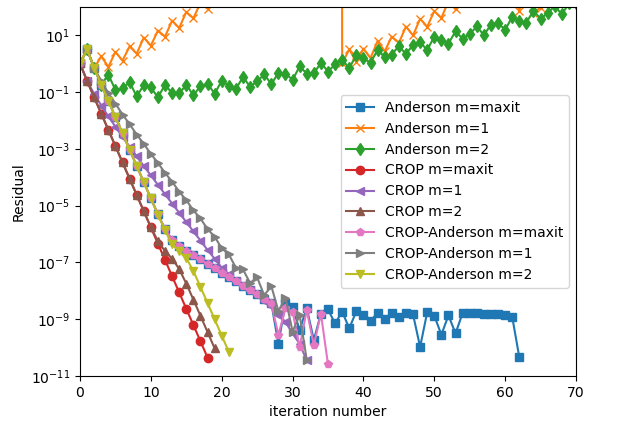}
    \caption{Convergence for a nonlinear problem in Example \eqref{exp:Ex4}.}
    \label{fig:nl1}
    \vspace{-0.25in}
\end{figure}
We can see that the first several steps of all methods are controlled by the linear part of the problem and can be compared with the results displayed in \Cref{fig:crop_lin1}. The difference occurs when the residual reaches $\approx 10^{-6}$. We can see that CROP algorithm, CROP(2), and CROP-Anderson(2) method converge in $18,19$ and $21$ steps, respectively, which is better than the other methods. It is worth mentioning that for CROP algorithm, although the control residuals get below the tolerance, the real residuals may not. Actually, in this example, $\displaystyle \|r_{CROP}^{(18)}\|_2=6.28\cdot 10^{-8}$, $\displaystyle \|r_{CROP(2)}^{(19)}\|_2=9.56\cdot 10^{-11}$ and $\displaystyle \|r_{CROP(1)}^{(32)}\|_2=5.19\cdot 10^{-11}$. In the case of CROP algorithm with small values of $\displaystyle  m$, the control residuals are good approximations of the real residuals. However, this does not have to be the case for the larger values of $\displaystyle m$. We will see some further disadvantages of the control residuals in Example \eqref{exp:Ex5}.

\begin{example}[A Small Nonlinear Problem]
\label{exp:Ex5}
\end{example}
Consider problem \eqref{eq:MainProblem}~\cite[Problem 2]{Hans2022,Hans2024} with
\[g\Big(\begin{bmatrix}
    x_1\\
    x_2
\end{bmatrix}\Big)=\frac{1}{2}
\begin{bmatrix}
    x_1+x_1^2+x_2^2\\
    x_2+x_1^2
\end{bmatrix} \ \mbox{ and exact solution } \ \begin{bmatrix}
    x_1^*\\
    x_2^*
\end{bmatrix} = \begin{bmatrix} 0 \\ 0 \end{bmatrix}.
\]
We set $maxit=100$, $tol=10^{-10}$ and run fixed-point iteration, Anderson Acceleration method, CROP algorithm and CROP-Anderson method with $\displaystyle x^{(0)} = [0.1,0.1]^T$ and different values of parameter $\displaystyle m$.
\begin{figure}[!ht]
    \begin{subfigure}[b]{0.49\textwidth}
        \centering
\includegraphics[scale=0.39]{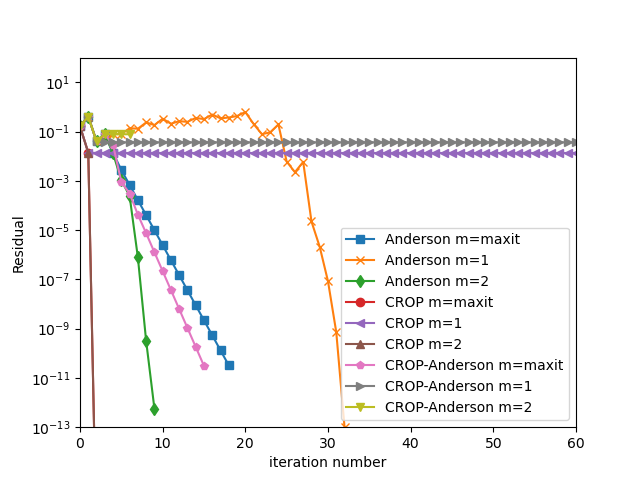}
    \caption{control residuals}
    \label{fig:crop_snl}
    \end{subfigure}
    \hfill
    \begin{subfigure}[b]{0.49\textwidth}
        \centering
    \includegraphics[scale=0.39]{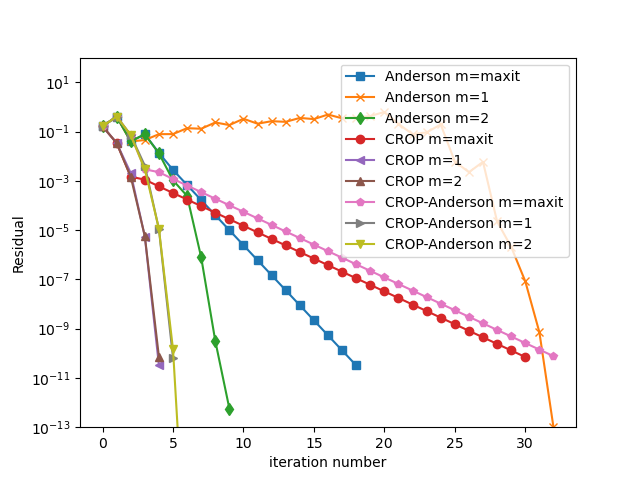}
    \caption{real residuals}
    \label{fig:rcrop_snl}
    \end{subfigure} 
    \caption{A small nonlinear problem in Example \eqref{exp:Ex5} with (a) control residuals and (b) real residuals.}
    \label{fig:snl}
\end{figure}
Note that all variants of CROP algorithm and CROP-Anderson method, except the full CROP-Anderson without truncation ($m=$maxit), are bad, see \Cref{fig:snl}. CROP algorithm and CROP(2) method break down in iteration $2$, but rCROP and rCROP-Anderson converge well. rCROP(1) and rCROP(2) method converge in $4$ iterations, while Anderson(2) method converges in $9$ and the Anderson(1) method in $32$ iterations. Even though using real residuals requires additional function evaluation in each iteration, Example \ref{exp:Ex5} illustrates that this extra cost is worthy as it reduces the number of iterations from $32$ in the case of Anderson Acceleration method to $8$ steps for rCROP algorithm.

\begin{example}[Bratu Problem]
\label{exp:Ex7}
\end{example}
In this example, we consider the Bratu Problem \cite[Section 5.1]{He2024}
\begin{equation}\label{eq:Bratu}
\begin{aligned}
\displaystyle
    \Delta u+\lambda e^u &= 0 \mbox{  in  } \Omega=(0,1)\times(0,1)\\
    u(x,y)&=0 \ \mbox{ for } \  (x,y)\in\partial\Omega
\end{aligned}.
\end{equation}
Using finite difference method with grid size $\displaystyle 100\times100$, the problem becomes
$\displaystyle Lx+h^2\lambda \exp(x)=0$, where $L$ is the $\displaystyle 10000\times 10000$ 2D Laplace matrix and $\displaystyle h=1/101$. We choose $\displaystyle \lambda=0.5$, $\displaystyle f(x) = Lx+h^2\lambda \exp(x)$, $\displaystyle g(x)=x+f(x)$ as well as parameters $maxit=400$, $tol=10^{-10}$, $x^{(0)} = 0$. We run \Cref{alg:Anderson}, \Cref{alg:CROP} and \Cref{alg:CROPAnd} with $\displaystyle m=\infty, 1, 2$. We also compare these methods with nlTGCR method introduced in \cite{He2024}. nlTGCR method is an extension of the Generalized Conjugate Residual (GCR)~\cite{Eisenstat1983} method to nonlinear problems by changing matrix $A$ to the Jacobian $-\Delta f$ in each step of the algorithm.
\begin{figure}[!b]
\vspace{-0.2in}
    \begin{subfigure}[b]{0.49\textwidth}
        \centering \includegraphics[width=1.0\textwidth,height=0.8\textwidth]{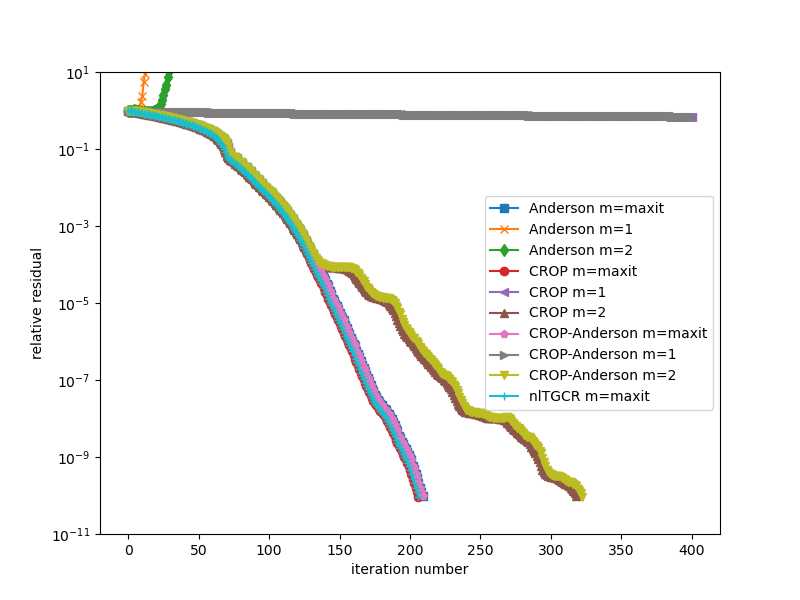}
    \label{fig:crop_Bratu}
    \caption{control residuals}
    \end{subfigure}
    \begin{subfigure}[b]{0.49\textwidth}
        \centering
\includegraphics[width=1.0\textwidth,height=0.8\textwidth]{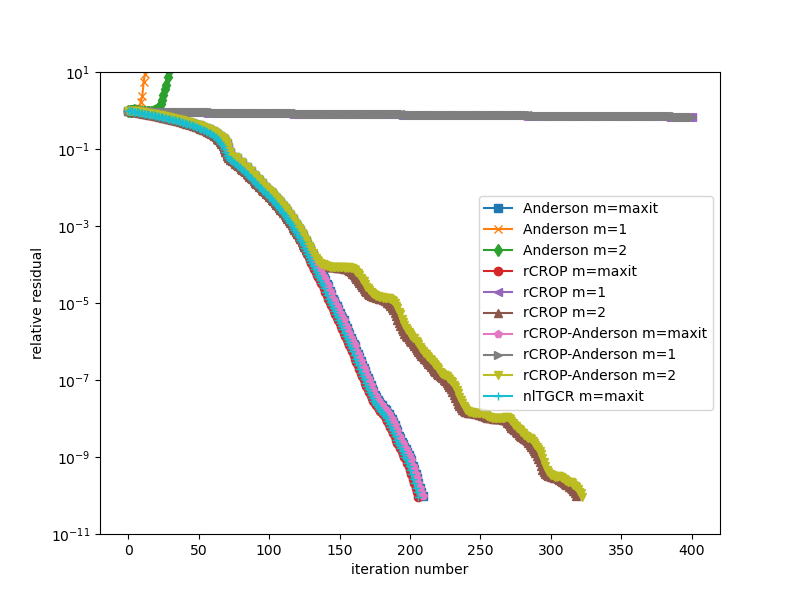}
    \label{fig:rcrop_Bratu}
    \caption{real residuals}
    \end{subfigure} 
    \caption{Bratu problem from Example \eqref{exp:Ex7} with (a) control residuals and (b) real residuals.}
    \label{fig:Bratu}
\end{figure}
Since the Bratu problem \eqref{eq:Bratu} is symmetric and has small nonlinearities, the rCROP results are almost the same as those of CROP algorithm, and confirm that $m=2$ is a good choice of the truncation parameter.

\section{Conclusions}
\label{sec:conclusions}
Based on the initial convergence analysis and experiments presented in this paper, CROP algorithm emerges as an interesting approach for linear and nonlinear problems with weak nonlinearities. Although for highly nonlinear problems CROP algorithm often fails to converge, its variant rCROP can behave even better than Anderson Acceleration method. Further theoretical and numerical studies of CROP algorithm and its variants, in particular for challenging large scale computational chemistry problems,  are the subject of an ongoing work. 

\vspace{-0.1in}
\section*{Acknowledgments}
The authors would like to thank Eric de Sturler, Tom Werner and Mark Embree for insightful discussions and helpful suggestions to this project. This work was supported by the National Science Foundation through
the awards DMS--2144181 and DMS--2324958.

\vspace{-0.1in}
\bibliographystyle{siamplain}
\bibliography{crop_paper}

\end{document}


\maketitle

This Supplementary Material provides some details regarding implementation and validation of discussed methods, as well as some additional numerical examples.


\section{Implementation details}

All numerical experiments presented in this section have been implemented in Julia 1.7.1~\cite{Julia-2017} and carried out on an Intel Xeon 8-Core 3.00GHz machine with 32GB memory. Here, we briefly discuss an adaptive variant of CROP algorithm for nonlinear problems that substitutes some rCROP steps with CROP steps when the approximation is good. Moreover, the idea of line search is discussed.

\paragraph{\textbf{Restarted and Adaptive Methods on Real Residuals}}

Since real residuals require additional function evaluations, rCROP algorithm usually takes longer to execute the same number of steps than CROP algorithm. Following the idea presented in~\cite{He2024}, an adaptive method is running an rCROP step every few iterations to check the accuracy of CROP algorithm step by measuring the angle $\displaystyle \theta^{(k)}$ between the control residual $\displaystyle f_C^{(k)}$ and the real residual $\displaystyle  r_C^{(k)}$ of that step, i.e.,
$$\displaystyle  \cos\theta^{(k)}=\frac{(f_C^{(k)},r_C^{(k)})}{\|f_C^{(k)}\|\|r_C^{(k)}\|}.$$
If $\displaystyle \cos\theta^{(k)}>0.99$, the nonlinearity of the problem is considered small, and the result produced by CROP algorithm can be viewed as a good approximation of the exact solution. Otherwise, the problem is highly nonlinear and should be solved with rCROP algorithm instead.

\paragraph{\textbf{Line Search Refinement}}






The line search is an iterative process used to find the solution of $f(x)=0$
when a descent search direction is available. One commonly used line search method is based on the Armijo condition~\cite{Arm66}. Since the line search requires computing residuals explicitly, when used it should be employed with rCROP algorithm instead of CROP. For rCROP algorithm, the function evaluation gives the real residuals and thus the line search is possible.

Since CROP gives a Broyden-type approximated inverse Jacobian
$$\displaystyle G_C^{(k+1)}=\mathscr{X}_C^{(k+1)}[(\mathscr{F}_C^{(k+1)})^T\mathscr{F}_C^{(k+1)}]^{-1}(\mathscr{F}_C^{(k+1)})^T,$$ we can do the line search with this approximated direction. However, for multisecant methods, we can not guarantee that the approximate Newton direction will be a descent direction, and therefore, a line search may fail. As the approximated direction can have a very large error, the Armijo condition may become meaningless. If the line search with approximated Jacobian fails, a better preconditioner is needed, or CROP algorithm without the line search must be used.

In fact, the line search works very badly for CROP algorithm. For linear and nonlinear problems with weak nonlinearities, rCROP has essentially the similar form as CROP algorithm, meaning that the obtained result is optimal in the search space and the explicit line search step is not needed.

Another way is to do a line search without derivatives, for example by using discrete line search methods like bisection or Li-Fukushima derivative-free line search. Li-Fukushima derivative-free line search~\cite{li2000} is designed to work with global convergent Broyden-type method, which update satisfies
\begin{equation}\label{eq:LFDFLS}
\displaystyle
    \sigma_0\|\Delta x^{(k)}\|^2\le (1+\eta_k)\|f(x^{(k)})\|-\|f(x^{(k+1)})\|,
\end{equation}
where $\sigma_0>0$ and $\sum\limits_{k=0}^\infty \eta_k\le \eta<\infty$.

Let $\displaystyle \Delta x^{(k)}=\lambda^{(k)} p^{(k)}$, where $\lambda^{(k)}$ is the step size, and $\displaystyle p^{(k)}$ is the search direction given as $\displaystyle  p^{(k)}=G_C^{(k+1)}\widetilde f_C^{(k+1)}$. Then $$\displaystyle x^{(k)}=\widetilde x_C^{(k+1)} \mbox{ and } x^{(k+1)}=x_C^{(k+1)}.$$





\section{Additional experimental results}


\begin{example}[A Real World Problem]
\label{exp:Ex2}
\end{example}
Let us now consider the same linear problem $Ax = b$ with a symmetric positive definite matrix \texttt{cfd1} from SuiteSparse Matrix Collection~\cite{DavH11} and vector $b\in\mathbb{R}^{70656}$ with its first entry $1$ and others $0$. The corresponding convergence results for different methods are shown in \Cref{fig:lin5}. Notice that choice of $m=2$ is the optimal for both CROP and CROP-Anderson algorithm.



\begin{figure}[!htb]
    \centering
    \begin{minipage}{0.48\textwidth}
        \centering
        \includegraphics[scale=0.38]{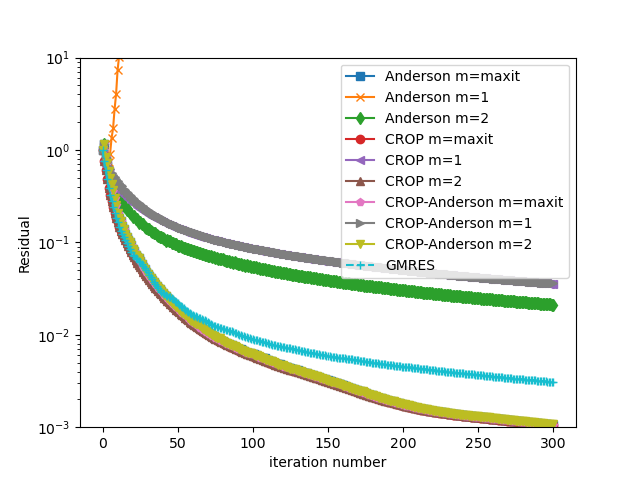}
        \caption{Convergence for a linear problem in Example \eqref{exp:Ex2}.}
    \label{fig:lin5}
    \end{minipage}%
    \hfill
    \begin{minipage}{0.48\textwidth}
        \centering
         \includegraphics[scale=0.38]{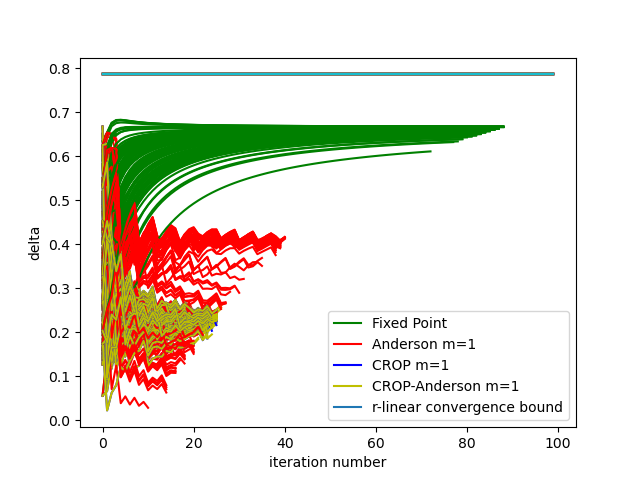}
         \caption{r-linear convergence factor for a small linear system in Example \eqref{exp:Ex3}.}
     \label{fig:slin_r}
    \end{minipage}
\end{figure}



\begin{example}[A Small Linear Problem]
\label{exp:Ex3}
\end{example}

In this example, we consider a linear system $Ax = b$ ~\cite[Problem 1]{Hans2022,Hans2024} with 
\begin{equation}
\label{eq:SLS}
A=\begin{bmatrix}1/3&-1/4\\0&2/3\end{bmatrix}\in\mathbb{R}^{2\times2} \quad \mbox{ and } \quad b=\begin{bmatrix}0\\0\end{bmatrix} \in \mathbb{R}^2.
\end{equation}

We choose $f :=b-Ax$ \ and \ $g := x+f=Mx$ \ with 
$$M=\begin{bmatrix}2/3&1/4\\0&1/3\end{bmatrix}\in\mathbb{R}^{2\times2}.$$

We set $maxit=100$, $tol=10^{-16}$ and run fixed-point iteration,
Anderson Acceleration, CROP, and CROP-Anderson algorithms
with $m=1$. The components $x_1^{(0)}$ and $ x_2^{(0)}$ of the initial vector $x^{(0)}$ are chosen randomly between $[-0.5,0.5]$. The r-linear convergence factors are shown in \Cref{fig:slin_r}. The damping parameters $\gamma^{(k)}$ for Anderson and CROP/CROP-Anderson methods are presented in \Cref{fig:slin_ba} and \Cref{fig:slin_bc}. Note that they have different oscillation behaviors. For Anderson Acceleration method, $\gamma^{(k)}$ display the same oscillation pattern independent on the choice of initial vector, whereas two different kinds of oscillations are
occurring for CROP algorithm.

%
\begin{figure}[!ht]
    \centering
    \begin{subfigure}[b]{0.49\textwidth}
        \centering
        \includegraphics[scale=0.39]{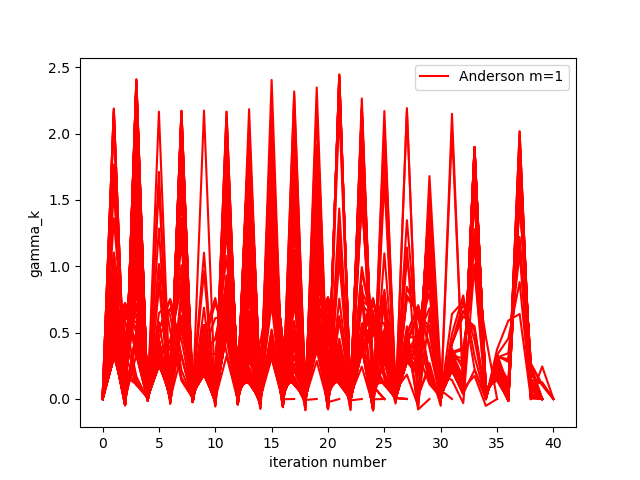}
        \caption{Anderson Acceleration}
        \label{fig:slin_ba}
    \end{subfigure}
    \begin{subfigure}[b]{0.49\textwidth}
        \centering
        \includegraphics[scale=0.39]{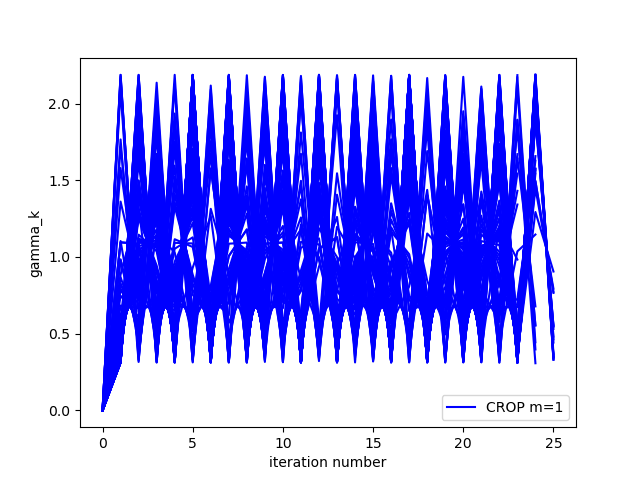}
        \caption{CROP Algorithm}
        \label{fig:slin_bc}
    \end{subfigure} 
    \caption{$\gamma_k$ for (a) Anderson Acceleration and (b) CROP Algorithm in Example~\eqref{exp:Ex3}.}
    \label{fig:slin_b}
\end{figure}
We show the r-linear convergence factors of the fixed point iteration, Anderson acceleration, CROP, and CROP-Anderson algorithms in \Cref{fig:slin_cf}, \Cref{fig:slin_ca}, \Cref{fig:slin_cc} and \Cref{fig:slin_cca}. We can see that the r-linear convergence factor is the function of the angle $\displaystyle \arctan\frac{x_2^{(0)}}{x_1^{(0)}}$.  The factor for CROP and CROP-Anderson oscillates much faster than the factor for Anderson acceleration.

\begin{figure}[!ht]
    \centering
    \begin{subfigure}[b]{0.49\textwidth}
        \centering
        \includegraphics[width=0.8\textwidth]{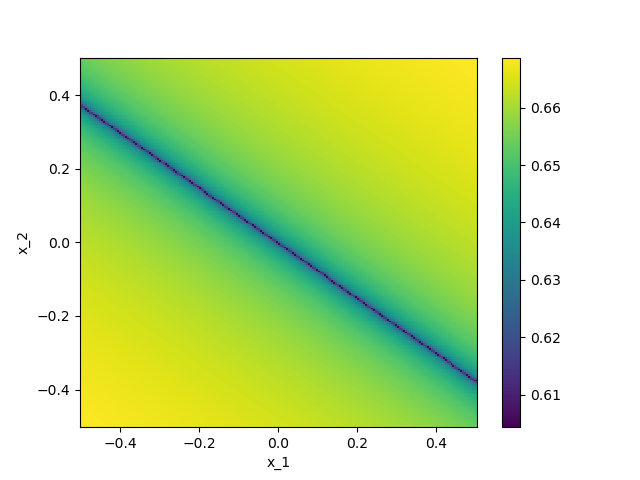}
        \caption{Fixed-point Iteration}
        \label{fig:slin_cf}
    \end{subfigure}
    \hfill
    \begin{subfigure}[b]{0.49\textwidth}
        \centering
        \includegraphics[width=0.8\textwidth]{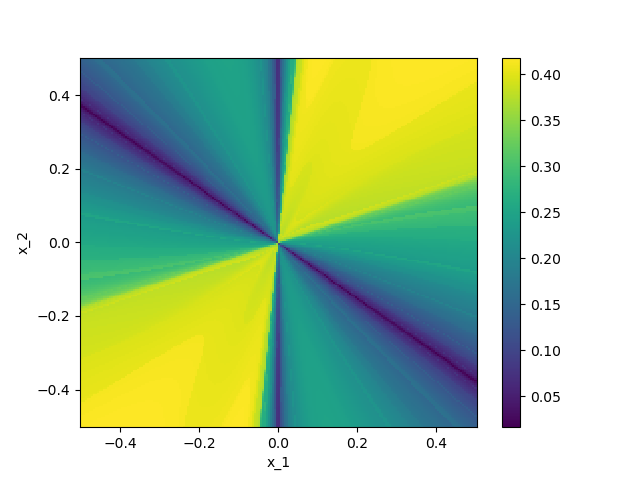}
        \caption{Anderson Acceleration}
        \label{fig:slin_ca}
    \end{subfigure} \\
    \begin{subfigure}[b]{0.49\textwidth}
        \centering
        \includegraphics[width=0.8\textwidth]{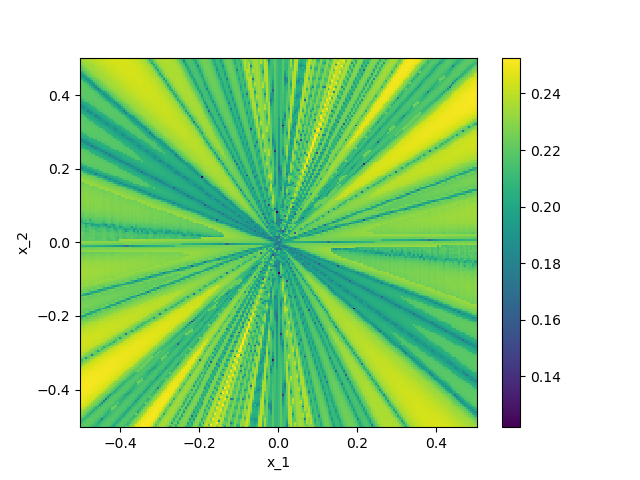}
        \caption{CROP Algorithm}
        \label{fig:slin_cc}
    \end{subfigure}
    \begin{subfigure}[b]{0.49\textwidth}
        \centering
        \includegraphics[width=0.8\textwidth]{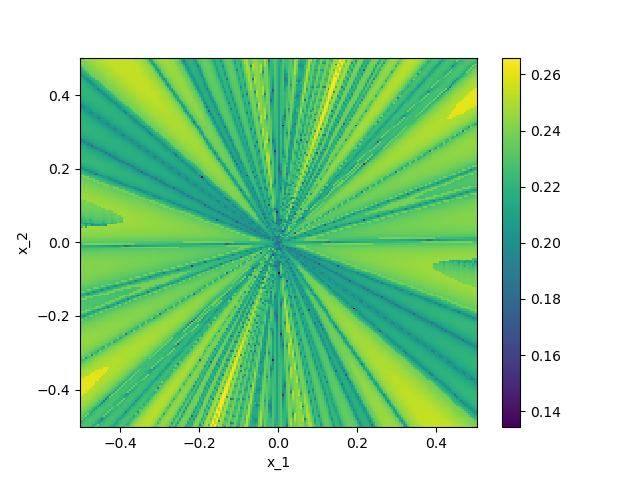}
        \caption{CROP-Anderson Algorithm}
        \label{fig:slin_cca}
    \end{subfigure} 
    \caption{r-linear convergence factors in Example \eqref{exp:Ex3}.}
    \label{fig:slin_c}
\end{figure}

\begin{example}[Nonlinear Eigenvalue Problem]
\label{exp:Ex6}
\end{example}
In this example, we consider a time-delay system with a distributed delay discussed in~\cite[Example 2]{Jarlebring2012} where the distributed term is a Gaussian distribution, i.e.,

\begin{equation}\label{eq:depsys}
    \begin{split}
    \dot x(t)=&\frac{1}{10}\begin{pmatrix}25&28&-5\\18&3&3\\-23&-14&35\end{pmatrix} x(t)
        +\frac{1}{10}\begin{pmatrix}17&7&-3\\-24&-21&-2\\20&7&4\end{pmatrix} x(t-\tau)\\
        &+\int\limits_{-\tau}^0\begin{pmatrix}14&-13&4\\14&7&10\\6&16&17\end{pmatrix}
        \frac{e^{(s+\frac12)^2}-e^{\frac14}}{10}x(t+s)ds
    \end{split}
\end{equation}

The eigenvalues of \eqref{eq:depsys} are given as the solutions of the nonlinear eigenvalue problem (NEP) $\displaystyle M(\lambda)v=0$ associated with a matrix-valued function
$$\displaystyle M(\lambda)=-\lambda I+A_0+A_1 e^{-\lambda\tau}+\int_{-\tau}^0 F(s) e^{\lambda s}ds,$$
where
$$\displaystyle A_0=\frac{1}{10}\begin{pmatrix}25&28&-5\\18&3&3\\-23&-14&35\end{pmatrix}, \quad A_1=\frac{1}{10}\begin{pmatrix}17&7&-3\\-24&-21&-2\\20&7&4\end{pmatrix},$$
$$F(s)=\begin{pmatrix}14&-13&4\\14&7&10\\6&16&17\end{pmatrix}
        \frac{e^{(s+\frac12)^2}-e^{\frac14}}{10}.$$
\noindent    
Consider $\displaystyle x=\begin{bmatrix}v&\lambda\end{bmatrix}^T$. Then, the nonlinear eigenvalue problem associated with \eqref{eq:depsys} can be written as a nonlinear equation of the form 
$$\displaystyle f(x) = f\Big(\begin{bmatrix}v\\\lambda\end{bmatrix}\Big) :=\begin{bmatrix}M(\lambda)v\\ c^Hv-1\end{bmatrix} = 0, $$
where $c$ is used to normalize the eigenvector $v$ and we choose $c$ the vector of ones here.
Let $\displaystyle g := x+\beta f$ with $\displaystyle \beta=0.1$. We run 
Anderson Acceleration, CROP, CROP-Anderson algorithms with 
$\displaystyle maxit=100$, $\displaystyle tol=10^{-10}$, $\displaystyle m=maxit$, $\displaystyle m=3$ and $\displaystyle m=5$. The initial vector $x^{(0)}$ is chosen as the vector of ones. The convergence of all methods is shown in \Cref{fig:dep}. 

\begin{figure}[!ht]
    \begin{subfigure}[b]{0.49\textwidth}
        \centering
\includegraphics[width=1.0\textwidth,height=0.8\textwidth]{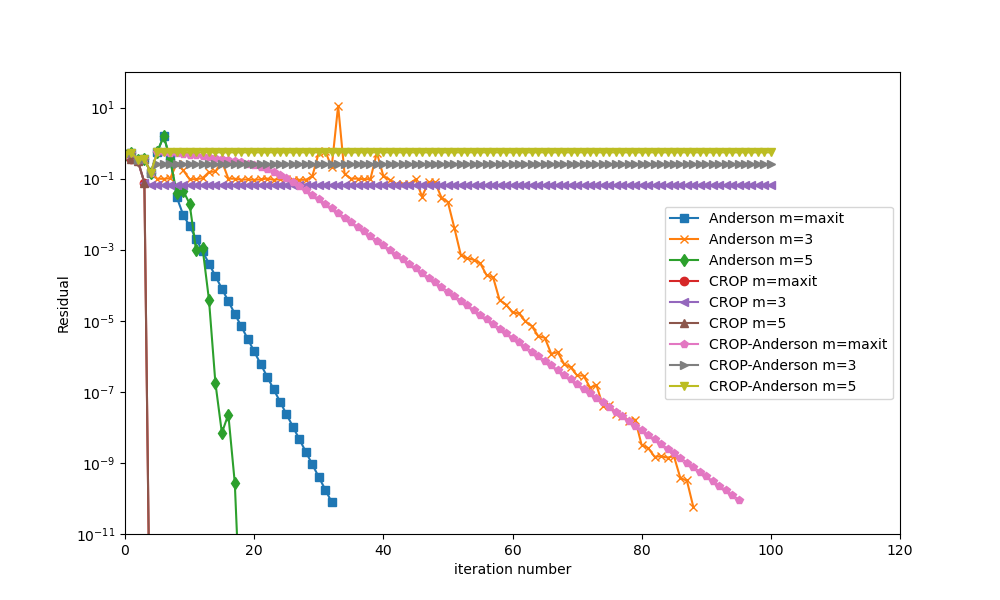}
    \caption{control residuals}
    \label{fig:crop_dep}
    \end{subfigure}
    \begin{subfigure}[b]{0.49\textwidth}
        \centering
\includegraphics[width=1.0\textwidth,height=0.8\textwidth]{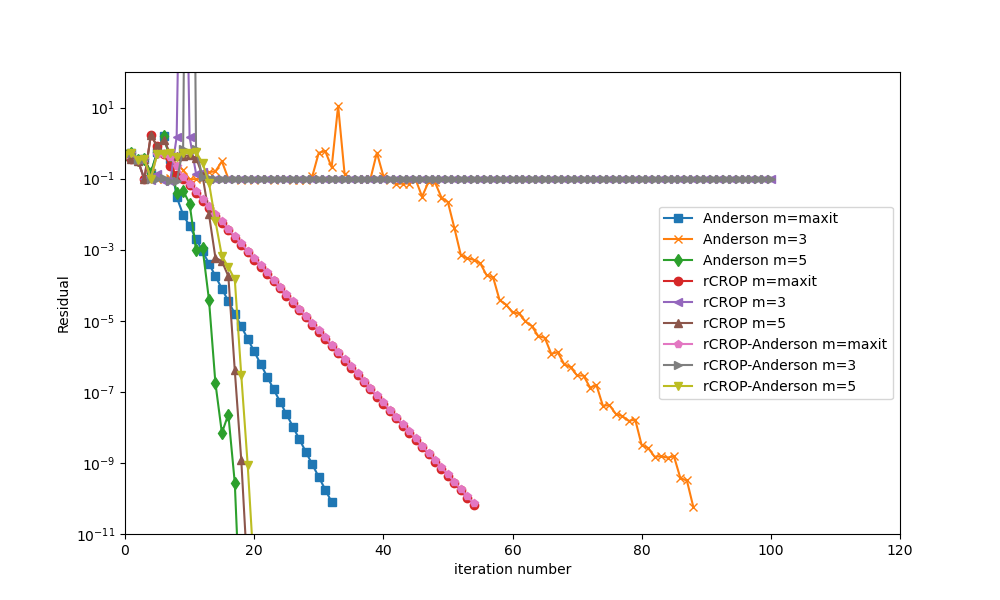}
    \caption{real residuals}
    \label{fig:rcrop_dep}
    \end{subfigure} 
    \caption{A nonlinear eigenvalue problem from Example \eqref{exp:Ex6} with (a) control residuals and (b) real residuals.}
    \label{fig:dep}
\end{figure}

\Cref{fig:crop_dep} illustrates the breakdown of CROP algorithm and the problems associated with using control residuals $\displaystyle f_C^{(k)}$. The eigenvector of this NEP is of length $3$, and the residual vector of dimension $n=4$. The problem is nonlinear and needs more than $n$ iterations to converge. When the least-squares problem 
for finding the coefficients of the CROP iterates is of size $m_C^{(k)}=4$, it has a solution and which makes $f_C^{(k)} = 0$. CROP algorithm for this problem breaks down at the $4$-th step. 
Meanwhile, rCROP and rCROP-Anderson method converge well when $m=3$ and $m=5$, respectively, see \Cref{fig:rcrop_dep}.






 







\bibliographystyle{siamplain}
\bibliography{crop_paper}